\documentclass[a4paper,10pt]{article}
\pdfoutput=1

\usepackage{geometry}
 \usepackage[T2A]{fontenc}
 \usepackage[russian,english]{babel}

\usepackage[latin1]{inputenc}
\usepackage{cite}
\usepackage{nameref,hyperref,url}
\usepackage[right]{lineno}
\usepackage{microtype}
\usepackage{todonotes}

\usepackage{graphicx,epstopdf}
\usepackage{amsmath,amsthm,amsfonts,amssymb}

\theoremstyle{plain}
\newtheorem{theorem}{Theorem}
\newtheorem{lemma}[theorem]{Lemma}

\newtheorem{proposition}[theorem]{Proposition}
\newtheorem{corollary}[theorem]{Corollary}
\theoremstyle{definition}

\theoremstyle{remark}
\newtheorem{remark}[theorem]{Remark}

\allowdisplaybreaks

\renewcommand{\O}{\Omega}
\newcommand{\R}{\mathbb{R}}
\newcommand{\N}{\mathbb{N}}

\renewcommand{\O}{\Omega}
\newcommand{\Od}{\Omega_{\delta}}
\newcommand{\Gnl}{\Gamma^{\text{nl}}}
\newcommand{\dx}{\,\mathrm{d}x}
\newcommand{\cD}{\mathcal{A}}
\newcommand{\cK}{\mathcal{K}}
\newcommand{\cU}{\mathcal{U}}
\newcommand{\cS}{\mathcal{S}}

\newcommand{\comploc}{\mathfrak{c}}
\newcommand{\compnl}{\mathfrak{c}_\delta}

\DeclareMathOperator{\supp}{supp}

\DeclareMathOperator{\diver}{div}
\DeclareMathOperator*{\minz}{minimize}
\DeclareMathOperator*{\argmin}{arg\,min}
\DeclareMathOperator{\diam}{diam}
\DeclareMathOperator{\cl}{closure}

\begin{document}

%\title{A nonlocal optimal design problem: well-posedness, convergence to a local limit problem and numerical approximation}

\title{Non-local control in the conduction coefficients: \\
well posedness and convergence to the local limit}

\author{Anton Evgrafov\textsuperscript{1,*} and Jos\'e C. Bellido\textsuperscript{2}
\bigskip
\\
\small
1 Department of Mechanical Engineering,
Technical University of Denmark
\\
\small
2 Department of Mathematics,
University of Castilla-La Mancha
\bigskip
\\
\small
* aaev@mek.dtu.dk}

\maketitle

\noindent  \paragraph*{Abstract:} We consider a problem of optimal distribution of conductivities in a system governed by a non-local diffusion law.  The problem stems from applications in optimal design and more specifically topology optimization.  We propose a novel parametrization of non-local material properties.  With this parametrization the non-local diffusion law in the limit of vanishing non-local interaction horizons converges to the famous and ubiquitously used generalized Laplacian with SIMP (Solid Isotropic Material with Penalization) material model.  The optimal control problem for the limiting local model is typically ill-posed and does not attain its infimum without additional regularization.  Surprisingly, its non-local counterpart attains its global minima in many practical situations, as we demonstrate in this work.  In spite of this qualitatively different behaviour, we are able to partially characterize the relationship between the non-local and the local optimal control problems.  We also complement our theoretical findings with numerical examples, which illustrate the viability of our approach to optimal design practitioners.

\noindent  \paragraph*{Keywords:} Nonlocal optimal design, nonlocal optimal control in the coefficients, convergence to local problems, numerical approximation of nonlocal problems.

% !TEX root = NonlocalOD.tex
%

%
\section{Introduction}
\label{sec:intro}

Nonlocal problems receive a lot of attention nowadays owing to a wide range of
applications they have in a variety of contexts. In particular, we mention micromechanics~\cite{Rogers91}, image processing~\cite{Gios08}, phase transitions~\cite{AlBe98}, pattern formation~\cite{Fife03}, population dispersal~\cite{CoCoElMa07}, optimal design~\cite{andres2015nonlocal} and shape optimization~\cite{bonder}, optimal control~\cite{Delia1} and inverse problems \cite{Delia2}; see also a very recent monograph on the subject~\cite{Du_book}. Among nonlocal problems, fractional or nonlocal diffusion plays a central role.  It has attracted enormous interest and a great deal of work has been done over the past twenty years.
The number of references on the subject and connections of nonlocal diffusion with
remarkable applications is really overwhelming. We refer the interested readers to the monographs~\cite{Mazon_book, Valdinoci_book} and the references therein.
In the context of continuum mechanics the non-local modelling goes back at least to the Eringen's model~\cite{eringen,MMS_2019},
and more recently has been focused on peridynamical modelling~\cite{kunin1975,silling2000reformulation,Peridynamics_book,Handbook,EmWe07, MengeshaDu}.
These latter models refrain from using the gradients of the state fields with the goal of
unified description of singular phenomena, such as fracture or dislocation.
When considering scalar equations, such as for example the steady state heat equation, peridynamics
equations may essentially be seen as a nonlocal diffusion equations on bounded domains~\cite{AkMe10,andres2015type}.

In this paper we consider a prototypical optimal design problem for diffusion
phenomena, in which one has to determine the best way of distributing
conducting materials inside a given computational domain.
For models governed by the local diffusion phenomena such problems have been studied for a long time
and are quite well understood, see for example~\cite{allaire2012shape,bendose2003topology,Pablo_book} and references therein.
As a very rough summary we can say that these problems are typically ill-posed and do not attain their infimum.
One possible way of dealing with this issue is to bring the limits of minimizing sequences into consideration and interpret them as composite materials obtained from mixing the original materials in the sense described by the theory of homogenization.
Another option is to restrict the set of considered material distributions by introducing constraints or penalty functions with regularizing effect.
An extremely successful example of the latter approach, which is nowadays widely utilized in engineering
practice, is the SIMP (Solid Isotropic Material with Penalization) material parametrization model combined
with additional regularization techniques~\cite{bendose2003topology,bendsoe1988generating,bourdin2001filters,sigmund1997design}.
An interesting recent study, which can be interpreted in the light of comparing the two outlined approaches
in the context of the steady state heat conduction, is~\cite{Suna}.

% The most famous example of the latter approach is t
% The latter approach has been successfully utilized in the
% Whereas the problem
% In the local framework, optimal design in conduction is nowadays a well understood topic, being the homogenization method the most popular and successful approach. An essential reference regarding the application of homogenization method to optimal design problems is \cite{allaire2012shape}. Another very popular way to deal with optimal design problems, mainly focused on looking for black and white designs (avoiding thus microstructural solutions as those provided by the homogenization method) through numerical algorithms is the topology optimization method, where SIMP (solid isotropic material with interpolation) approach plays a central role \cite{bendose2003topology}. A recent interesting comparison on the effectiveness of the topology optimization and homogenization methods is \cite{Suna}. For optimal design problems for scalar equations like this one, another interesting and successful approach based on vector variational problems was also recently developed \cite{Pablo_book}.% succeeding even in dealing with problems with non-weakly continuous cost functionals \cite{}, or problems involving pointwise constraints on the  design and state variables \cite{bellido1,bellido2}.

Research in optimal design of systems described by the non-local governing equations is in its infancy.
Our present model is inspired by the recent studies~\cite{andres2015nonlocal, andres2017convergence},
where a very similar optimal design problem for a nonlocal diffusion state law was analyzed.
The main novelty of our work is the way in which the material properties (conductivities) enter into the model. We chose a nonlocal interpolation of material properties that allows us to establish a natural link between the nonlocal optimal design problem and a local one in which material properties are interpolated by the SIMP scheme. Our main objective is to analyze this \emph{nonlocal} optimal design problem and to characterize its relationship with the local optimal design problem as the nonlocal parameter (the \emph{interaction horizon}) goes to zero.
We have documented our preliminary findings in a brief note~\cite{eb2019sensitivity}, where we have been primarily
concerned with the qualitative relation between our non-local model and a particular heuristic
regularization method (``sensitivity filtering'', see~\cite{sigmund1997design}) for the local model with SIMP.
This work includes the technical results, proofs, and numerical experiments, which have been omitted
from~\cite{eb2019sensitivity} owing to the space requirements.

The outline of the paper is as follows.
In Section~\ref{sec:model} we formulate the non-local state equations and discuss their well-posedness.
We also state the nonlocal optimal design problem, which is going to be the main subject of our study, as well as its local counterpart.
Section~\ref{sec:model} also includes Theorem~\ref{thm:dsgn2state}, an interesting result
illustrating the continuity of the conductivity-to-state map for the non-local problem.
In Section~\ref{sec:exist} we include several existence results for the nonlocal optimal design problem depending on the SIMP penalization parameter.
In Section~\ref{sec:gamma} we address the convergence of the nonlocal problems to the local one and the connection of our proposed model with SIMP in the local case.
Finally, in Section~\ref{sec:num}, a numerical approach to the nonlocal optimal design problems is described and several numerical examples are presented.

%% FIXIT

%% FIXIT

% !TEX root = NonlocalOD.tex
%

%
\section{Problem formulation and perliminaries}
\label{sec:model}

\subsection{Local control in the conduction coefficients}
\label{subsec:locprob}
We begin our discussion with a well understood optimal control problem in the
conduction coefficients of generalized Laplace equation, or topology optimization
through material distribution, see for example~\cite{cea1970example,allaire2012shape,bendose2003topology}
and references therein.
This problem will serve us both as a reference physical model and also as a limiting problem
later on.

Let \(\O \subset \R^n\), \(n\ge 2\), be an open, bounded, and connected domain.
In this domain we consider the generalized Laplace equation with a volumetric source
\(f \in L^2(\O)\), homogeneous Dirichlet boundary conditions, and a spatially heterogeneous
diffusion (conduction) coefficient \(\kappa^{\text{loc}}\).
Its weak solution \(u \in \cU_0 = W^{1,2}_0(\O)\) satisfies the variational statement
\begin{equation}\label{eq:w_loc}
  a_{\kappa^{\text{loc}}}(u,v) = \ell(v), \qquad \forall \: v \in \cU_0,
\end{equation}
where the parametric bilinear form \(a_{\kappa^{\text{loc}}}\) and the linear functional
\(\ell\) are given by
\begin{equation}\label{eq:al_loc}
  \begin{aligned}
    a_{\kappa^{\text{loc}}}(u,v) &= \int_{\O} \kappa^{\text{loc}}(x) \nabla u(x)\cdot \nabla v(x)\dx,\quad \text{and}\\
    \ell(v) &= \int_{\O} f(x)v(x)\dx.
  \end{aligned}
\end{equation}
We recall that
\begin{equation}\label{eq:loc_dir_pple}
  u \text{\ solves~\eqref{eq:w_loc}} \quad\iff\quad
  u = \argmin_{v \in \cU_0} I_{\kappa^{\text{loc}}}(v) := \bigg[\frac{1}{2} a_{\kappa^{\text{loc}}}(v,v) - \ell(v)\bigg],
\end{equation}
where \(I_{\kappa^{\text{loc}}}\) is an associated quadratic energy functional.
We will utilize the shorthand notation \(\cS^{\text{loc}}(\kappa^{\text{loc}})
:= \argmin_{v \in \cU_0} I_{\kappa^{\text{loc}}}(v)\) for the control coefficient-to-state
operator for this system.

In connection with this governing equation we consider a problem of optimal distribution
of conductive material in \(\O\) under simple constraints.
Specifically, we will define the following convex set of admissible material distributions:
\begin{equation}\label{eq:adm_loc}
  \cD = \bigg\{\, \rho \in L^\infty(\O) \mid \underline{\rho}\le
  \rho(\cdot)\le \overline{\rho}, \text{a.e.\ in \(\O\)},
  \int_{\O} \rho(x)\dx \le \gamma|\Omega| \,\bigg\},
\end{equation}
where \(\gamma \in ]\underline{\rho},\overline{\rho}[\) is a given volume
fraction, and \(0 < \underline{\rho} < \underline{\rho} < \infty\) are given bounds.\footnote{%
Our setup can be easily generalized to the situation when
\(\underline{\rho}, \overline{\rho} \in L^\infty(\O)\).}
We assume that the local material conductivity \(\kappa^{\text{loc}}\) is related
to the control parameter \(\rho \in \cD^\delta\) through the so-called
SIMP (Simple Interpolated Material with Penalization) model, see~\cite{bendose2003topology,allaire2012shape}:
\(\kappa^{\text{loc}}(x) = \rho^p(x)\), where \(p \ge 1\) plays
a role of a penalty parameter in certain optimal design problems.
Each \(\rho\in \cD^\delta\) is therefore mapped into \(\kappa^{\text{loc}}\)
satisfying the bounds
\(\underline{\kappa} = \underline{\rho}^p \le \kappa^{\text{loc}} \le \overline{\kappa} = \overline{\rho}^p\).

For a given performance functional \(J: \cD \times \cU_0 \to \R\) we consider
the following optimization problem:
\begin{equation}\label{eq:min_loc}
  \minz_{\rho \in \cD} J(\rho,\cS^{\text{loc}}(\rho^{p})).
\end{equation}
Note that the use of ``\(\minz\)'' instead of ``\(\inf\)'' does not automatically mean
that the infimum is attained in~\eqref{eq:min_loc}.
In fact, in most interesting situations this is not the case without additional
regularization of the problem~\eqref{eq:min_loc}.
We will return to this issue in Section~\ref{sec:exist}.

Of particular interest to us will be the case of \(J(\rho,u) = \ell(u)\),
corresponding to \emph{compliance minimization}.
Note that in this case the \emph{reduced} compliance can be expressed
in a variety of ways:
\begin{equation}
  \comploc(\rho) := \ell(\cS^{\text{loc}}(\rho^{p})) =
  a_{\rho^p}(\cS^{\text{loc}}(\rho^{p}),\cS^{\text{loc}}(\rho^{p}))
  = -2 I_{\rho^p}(\cS^{\text{loc}}(\rho^{p})).
\end{equation}
The last expression in conjunction with~\eqref{eq:loc_dir_pple} allows us
to state the compliance minimization problem as a saddle point problem, which
will be used later on.
\subsection{Non-local state equation}
We will now introduce a non-local analogue of the local governing equations~\eqref{eq:w_loc}, \eqref{eq:al_loc}.
We will use \(\delta>0\) to denote the non-local interaction horizon,
and \(B(x,\delta) = \{\, y \in \R^n : |x-y| < \delta \,\}\) to denote an open ball
of radius \(\delta\) centered at \(x \in \R^n\).
Let \(\Od = \cup_{x \in \O} B(x,\delta)\) be the set of points located within distance \(\delta\)
from points in \(\O\), and \(\Gnl = \Od\setminus\O\) be the ``non-local boundary'' of \(\O\).
We consider nonlocal linear diffusion equations, that in this scalar framework can also be seen as peridynamics equations.
Similarly to~\eqref{eq:w_loc}, they will be formulated with the help of the following parametric bilinear form
\(a_{\delta,\kappa}(\cdot,\cdot)\):
\begin{equation}\label{eq:al}
  \begin{aligned}
    a_{\delta,\kappa}(u,v) &=
    \int_{\Od}\int_{\Od} \kappa(x,x')
    A_{\delta}(|x-x'|)\frac{u(x)-u(x')}{|x-x'|}\frac{v(x)-v(x')}{|x-x'|}\dx\dx',
  \end{aligned}
\end{equation}
while we use the same linear functional \(\ell(\cdot)\) defined in~\eqref{eq:al_loc}.
In the equations above, \(\kappa \in L^\infty(\Od\times\Od)\) is a ``nonlocal conductivity'', and
\(A_{\delta}:\R_+\to \R_+\) is a radial kernel satisfying certain conditions.
More precisely, we assume that
\begin{equation}\label{eq:adm_kappa}
  \kappa \in \cK^\delta = \{\, \tilde{\kappa} \in L^\infty(\Od\times\Od)
  \mid \underline{\kappa}\le
  \tilde{\kappa}(x,x')=\tilde{\kappa}(x',x) \le \overline{\kappa}, \forall x,x'\in \Od \,\},
\end{equation}
for some \(0< \underline{\kappa} < \overline{\kappa} < +\infty\)\footnote{%
Dependence on \(\delta\) in the definition of the set \(\cK^\delta\) could be
omitted by considering nonlocal conductivities defined on the whole space
\(\R^n\times\R^n\).%
}, and that \(A_\delta\) satisfies the following conditions:
\begin{equation}\label{eq:kernelnrm}
  \frac{1}{n}\int_{\R^n} A_\delta(|x|)\dx = 1,
\end{equation}
\begin{equation}\label{eq:smallsupport}
   \supp A_\delta(|\cdot|) \subset B(0,\delta), \quad \forall \; \delta>0,
\end{equation}
and there exists \(s \in (0,1)\) and \(c_{\delta}>0\) such that
\begin{equation}\label{fractionallb}
    A_\delta(|x|) \ge \frac{c_{\delta}}{|x|^{n+2s-2}}, \quad\forall \; x \in B(0,\delta/2) \setminus\{0\},
    \delta>0.
\end{equation}

Let \(\cU^\delta = \{\, u \in L^2(\Od) \mid a_{\delta,1}(u,u) < +\infty\,\}\),
and \(\cU_0^\delta\) be the topological closure of \(C^\infty_c(\O)\)
(where we extend the functions in \(C^\infty_c(\O)\) by \(0\) outside of \(\O\))
in \(\cU_\delta\) with respect to the semi-inner product \(a_{\delta,1}(\cdot,\cdot)\).
%%, i.e. the set of functions \(u\) such that there exists a sequence \(u_j\in C^\infty_0(\O)\) such that
%%\[a_{\delta,1}(u_j-u,u_j-u)\rightarrow 0,\mbox{  as }j\rightarrow +\infty.\]
Since the nonlocal conductivities are uniformly bounded from above and away from zero,
for any \(\kappa \in \mathcal{K}^\delta\) and \(u \in \cU^\delta\) we have the estimates
\begin{equation*}
  \underline{\kappa}a_{\delta,1}(u,u) \le a_{\delta,\kappa}(u,u) \le \overline{\kappa}a_{\delta,1}(u,u).
\end{equation*}
Consequently, \(u\in {\cU}^\delta\) if and only if \(\|u\|_{L^2(\Od)} < +\infty\)
and \(a_{\delta,\kappa}(u,u)<+\infty\).
In this notation, the non-local governing equations we will study can be
stated as follows: find \(u\in \cU_0^\delta\) such that
\begin{equation}\label{eq:w}
  a_{\delta,\kappa}(u,v) = \ell(v), \quad \forall v \in \cU_0^\delta.
\end{equation}

\subsection{Non-local control in the conduction coefficients}
We will now introduce a non-local analogue of~\eqref{eq:min_loc},
where the local governing equations~\eqref{eq:w_loc}, \eqref{eq:al_loc} are replaced
with their non-local analogue introduced in the previous section.
Similarly to~\eqref{eq:adm_loc}, we define the following convex set of admissible
material distributions:
\begin{equation}\label{eq:adm}
  \cD^\delta = \bigg\{\, \rho \in L^\infty(\Od) \mid \underline{\rho}\le
  \rho(\cdot)\le \overline{\rho}, \text{a.e.\ in \(\Od\)},
  \int_{\O} \rho(x)\dx \le \gamma|\Omega| \,\bigg\},
\end{equation}
where the parameters are exactly as in~\eqref{eq:adm_loc}.
We still assume that the local material conductivity \(\kappa^{\text{loc}}\) satisfies
the SIMP model \(\kappa^{\text{loc}}(x) = \rho^p(x)\), \(x\in\Od\).
In addition, we assume that the non-local conductivity \(\kappa(x,x')\) entering~\eqref{eq:w}
is simply a geometric mean of \(\kappa^{\text{loc}}(x)\) and \(\kappa^{\text{loc}}(x')\),
that is,
\(\kappa(x,x') = \sqrt{\kappa^{\text{loc}}(x)\kappa^{\text{loc}}(x')} = \rho^{p/2}(x)\rho^{p/2}(x')\).
For a given performance functional \(J: \cD^\delta \times \cU^\delta \to \R\) we consider
the following optimization problem:
\begin{equation}\label{eq:min}
  \begin{aligned}
    &\minz_{(\rho,u) \in \cD^\delta\times\cU_0^\delta} J(\rho,u),\\
    &\text{subject to \(u\) solving~\eqref{eq:w} with \(\kappa(x,x') = \rho^{p/2}(x)\rho^{p/2}(x')\)}.
  \end{aligned}
\end{equation}
We will still refer to the case \(J(\rho,u) = \ell(u)\) as the \emph{compliance minimization}
problem.
\begin{remark}
  Essentially, the only difference between the non-local optimal design problem
  considered here and the one in~\cite{andres2015nonlocal} is the way the
  control variables enter the bilinear form.
  In order to preserve the symmetry of the form with respect to \((x,x')\)
  one takes an average between the local conductivities at \(x\) and \(x'\) to be
  the non-local conductivity in the bilinear form defining the non-local governing
  equations.
  The authors of~\cite{andres2015nonlocal} take an arithmetic mean, whereas
  we use a geometric mean (for \(p=1\)) for the same purpose.
  Advantages of this choice will become clear when we discuss the relationship
  between the nonlocal problem described in this section and the incredibly
  popular and successful SIMP method for (local) optimal design outlined in Section~\ref{subsec:locprob}.
\end{remark}

\subsection{Well-posedness and continuity of the conductivity-to-state operator}
In this subsection we recall a couple of known results about the non-local equation~\eqref{eq:w}.

Under the assumptions~\eqref{eq:kernelnrm}--\eqref{fractionallb},
\(a_{\delta,1}(\cdot,\cdot)\) defines an inner product on \(\mathcal{U}_0^\delta\)
thereby making it a Hilbert space~\cite{andres2015type}.
We denote by \(\|\cdot \|_{\cU_0^\delta}\) the norm induced by this inner product.
Furthermore, owing to~\eqref{fractionallb}, there are constants
\(\tilde{c}_{\delta}>0\), \(\hat{c}_{\delta}>0\) independent from \(u\)
such that for any \(\delta>0\), \(\kappa \in \cK^\delta\) and
\(u\in \mathcal{U}_0^\delta\) we have the inequalities
\begin{equation}\label{eq:L2coerc}
  \begin{aligned}
    a_{\delta,\kappa}(u,u)
    &\ge
    c_{\delta_0} \underline{\kappa}\int_{\Od} \int_{\Od\cap B(x,\delta/2)} \frac{(u(x)-u(x'))^2}{|x-x'|^{n+2s}} \dx'\dx \ge
    \tilde{c}_{\delta} \underline{\kappa}\int_{\Od} \int_{\Od} \frac{(u(x)-u(x'))^2}{|x-x'|^{n+2s}}\dx'\dx
    \\&\ge
    \hat{c}_{\delta} \underline{\kappa}\|u\|^2_{L^2(\O)},
  \end{aligned}
\end{equation}
where the second inequality is a consequence of \cite[Proposition~6.1]{bellido2014} and
the third inequality is established in~\cite[Lemma~6.2]{bellido2014}.
When combined with Cauchy--Bunyakovsky--Schwarz inequality, \eqref{eq:L2coerc} implies that
\(\ell: \cU_0^\delta \to \R\) is a continuous linear functional with respect to $\| \cdot\|_{{\cU}_0^\delta}$.
Consequently we have verified the necessary assumptions for the application of Lax--Milgram Lemma
(see, for example, \cite{brezis}) allowing us to conclude the following.
\begin{theorem}\label{thm:existence}
  For each \(\delta \in ]0,\delta_0[\), \(\kappa \in \cK^\delta\) and \(f \in L^2(\O)\) there is
  a unique solution \(u \in \cU_0^\delta\) of~\eqref{eq:w}.
  Furthermore, this solution minimizes the associated quadratic energy functional:
  \begin{equation}\label{eq:nl_dir_pple}
    u = \argmin_{v \in \cU_0^\delta} I_{\delta,\kappa}(v):=
    \bigg[
    \frac{1}{2}a_{\delta,\kappa}(v,v) - \ell(v)
    \bigg].
  \end{equation}
\end{theorem}
This well-posedness result allows us to univocally define the coefficient-to-state operator
\(\cS_\delta(\kappa):=\argmin_{v \in \cU_0^\delta} I_{\delta,\kappa}(v)\), exactly as in the
local case.
Furthermore, we put the \emph{reduced} non-local compliance to be
\begin{equation}
  \begin{aligned}
  \compnl(\rho) &:= \ell(\cS_\delta(\kappa)) =
  a_{\kappa}(\cS_\delta(\kappa),\cS_\delta(\kappa)) = -2 I_{\delta,\kappa}(\cS_\delta(\kappa)),
\end{aligned}
\end{equation}
where \(\kappa(x,x') = \rho^{p/2}(x)\rho^{p/2}(x')\).

Nonlocal diffusion and peridynamics equations have been studied extensively.
Existence results for a much more general class of nonlinear and nonlocal variational principles
have been obtained in~\cite{bellido2014}.
Particularly, well-posedness of equation~\eqref{eq:w} has been shown
in~\cite[Theorem~1.2]{andres2015type} for slightly more general conditions on the kernel \(A_\delta\).
More specifically, hypothesis~\eqref{fractionallb} could be further relaxed without sacrificing
the conclusions of Theorem~\ref{thm:existence}.
Even though our assumption~\eqref{fractionallb} is not optimal if one is only concerned with solvability of
the non-local governing equations, it is convenient for us since it implies the compact embedding
of \(\mathcal{U}_0^\delta\) into \(L^2(\Od)\), and ultimately a certain continuity
of the coefficient-to-state operator for the non-local equation~\eqref{eq:w}.

Indeed, let us revisit the string of inequalities~\eqref{eq:L2coerc}.
Note that the third term in~\eqref{eq:L2coerc} is nothing else but the Gagliardo
seminorm \(|u|_{W^{s,2}(\Od)}^2\).
Therefore, in addition to the continuous embedding of \(\mathcal{U}_0^\delta\) into
\(L^2(\Od)\), we also have a continuous embedding of \(\mathcal{U}_0^\delta\) into
the fractional Sobolev space \(W^{s,2}(\Od)\).
Since the latter space is compactly embedded into \(L^2(\Od)\)
(see~\cite[Theorem~7.1]{di2012hitchhikers}),
we also have a compact embedding of \(\mathcal{U}_0^\delta\) into  \(L^2(\Od)\).

Let us now discuss the conductivity-to-state operator for the non-local equation~\eqref{eq:w}.
In the case of the classical, local diffusion (generalized Laplace) equation,
the conductivity-to-state operator \(\cS^{\text{loc}}\) is famously not
continuous with respect to weak\(^*\) convergence of conductivities in \(\cK^\delta\).
In fact a stronger \(H\)-convergence, or \(G\)-convergence in our self-adjoint case,
has been specifically defined to obtain such a result~\cite{Murat1,Murat2, spagnolo1967sul,Tartar1,Tartar2}.
The notion of \(H\)-convergence has been recently extended to the nonlocal \(p\)-Laplacian
in~\cite{bonder_Hconv}, proving its sequential compactness for bounded coefficients.
In stark contrast, the conductivity-to-state operator \(\cS\) for the non-local
equation~\eqref{eq:w} is continuous with respect to weak\(^*\) convergence of conductivities
in \(\cK^\delta\).
This fact has been established in~\cite[Theorem~6]{andres2015nonlocal}.
For the sake of completeness and clarity we include a simpler proof of this result, which is in the same spirit as the simple characterization of $H$-convergence for fractional $p$-Laplacian equations given by the authors in~\cite{H-convergence}.
\begin{theorem}\label{thm:dsgn2state}
 Let us fix \(\delta \in ]0,\delta_0[\) and \(f \in L^2(\O)\).
 Consider a sequence of conductivities \(\kappa_j \in \cK^\delta\), together
 with the corresponding sequence \(u_j = \cS_\delta(\kappa_j) \in \cU_0^\delta\)
 of solutions to~\eqref{eq:w}.
 Assume that \(\kappa_j \rightharpoonup \hat{\kappa}\), weak\(^*\) in
 \(L^\infty(\Od\times\Od)\), and let \(\hat{u} = \cS(\hat{\kappa}) \in \cU_0^\delta\)
 be the corresponding solution to~\eqref{eq:w}.
 Then, we have
 \(\lim_{j\to\infty} \|u_{j}-\hat{u}\|_{L^2(\Od)} =\lim_{j\to\infty} \|u_{j}-\hat{u}\|_{\cU_0^\delta} = 0\).
\end{theorem}
\begin{proof}
  Recall that \(\hat{u}\) is the unique minimizer of \(I_{\delta,\hat{\kappa}}\),
  and \(u_j\) is that of \(I_{\delta,\kappa_j}\).
  Keeping in mind that \(\kappa_j\rightharpoonup \hat{\kappa}\), weak\(^*\) in
  \(L^\infty(\Od\times\Od)\), and the inclusion
  \begin{equation*} A_\delta(|x-x'|)\frac{(\hat{u}(x)-\hat{u}(x'))^2}{|x-x'|^2} \in L^1(\Od\times \Od)\end{equation*}
  we have the inequality
  \begin{equation}\label{eq:conv1}
    \limsup_{{j}\to\infty} I_{\delta,\kappa_{j}}( u_{j}) \le
    \lim_{{j}\to\infty} I_{\delta,\kappa_{j}}(\hat{u})= I_{\delta,\hat{\kappa}}(\hat{u}).
  \end{equation}

  Let us now extract a subsequence \(\{(\kappa_{j'},u_{j'})\}\), \({j'}=1,2,\dots\) from the
  original sequence such that \(\liminf_{j\to\infty} I_{\delta,\kappa_{j}}( u_{j})
  =\lim_{j'\to\infty} I_{\delta,\kappa_{j'}}( u_{j'})\).
  Note that as a direct consequence of~\eqref{eq:L2coerc} we get an uniform estimate
  \(\underline{\kappa}\|u_{j}\|_{\cU_0^\delta}^2 \le
  a_{\delta,\kappa_j}(u_j,u_j) = \ell(u_j) \le \hat{c}_{\delta}^{-1/2}(\overline{\kappa}/\underline{\kappa})^{1/2}\|f\|_{L^2(\O)}
  \|u_{j}\|_{\cU_0^\delta}\), \(j=1,2,\dots\)
  Therefore there exists \(u\in \mathcal{U}_0^\delta\) and a further subsequence,
  labelled by \(\{u_{j''}\}\),
  \(j''=1,2,\dots\), such that
  \begin{equation*}
    u_{j''}\rightharpoonup u, \text{weakly in\ } \mathcal{U}_0^\delta,
    \quad
    u_{j''}\rightarrow u, \text{ strongly in\ } L^2(\Od),
    \text{\ and\ }
    u_{j''}(x)\rightarrow u(x), \text{\ a.e.\ in \(\Od\)}.
  \end{equation*}
Let us define the finite measures
\begin{equation*}
  \begin{aligned}
    \mu_j(E)&=\int_E\kappa_j(x,x')\dx\dx'=\int_{\Od\times\Od}\chi_{E}(x,x')\kappa_j(x,x')\dx\dx',
    \quad j=1,2,\dots \quad\text{and}\\
    \hat{\mu}(E)&=\int_E\hat{\kappa}(x,x')\dx\dx'=\int_{\Od\times\Od}\chi_{E}(x,x')\hat{\kappa}(x,x')\dx\dx',
  \end{aligned}
\end{equation*}
where \(E\subset \Od\times\Od\) is an arbitrary Lebesgue measurable set
and \(\chi_E \in L^1(\Od\times\Od)\) is its characteristic function.
Weak\(^*\) convergence of \(\kappa_j\) to \(\hat{\kappa}\) implies the
strong convergence of measures \(\lim_{j\to\infty} \mu_j = \hat{\mu}\) (i.e. \(\lim_{j\rightarrow \infty} \mu_j(E)=\hat{\mu}(E)\) for any measurable set \(E\subset (\Od\times\Od)\)).
Furthermore, in view of~\eqref{eq:conv1} and the continuity of \(\ell\) the non-negative sequence
\(a_{\delta,\kappa_{j}}(u_{j},u_{j}) =
2 [I_{\delta,\kappa_{j}}(u_{j}) + \ell(u_{j})]\), \({j}=1,2,\dots\)
is bounded from above.
Therefore we can apply the generalized Fatou's lemma \cite[Prop. 17, pg. 269]{royden}:
\begin{equation*}
  \begin{aligned}
  \liminf_{{j}\to\infty} I_{\delta,\kappa_{j}}(u_{j})
  &=
  \lim_{{j''}\to\infty}\bigg[ \frac{1}{2}\int_{\Od\times\Od} A_\delta(|x-x'|)\frac{(u_{j''}(x)-u_{j''}(x'))^2}{|x-x'|^2}
  \,\mathrm{d}\mu_{j''}\bigg] - \lim_{{j''}\to\infty}\ell(u_{j''})
  \\& \ge
  \frac{1}{2}\int_{\Od\times\Od} A_\delta(|x-x'|)\frac{(u(x)-u(x'))^2}{|x-x'|^2}
  \,\mathrm{d}\hat{\mu} - \ell(u)
  =I_{\delta,\hat{\kappa}}(u) \ge I_{\delta,\hat{\kappa}}(\hat{u}).
\end{aligned}
\end{equation*}
In particular, \(\lim_{j\to\infty} I_{\delta,\kappa_j}(u_j)= I_{\delta,\hat{\kappa}}(\hat{u})\).
The strong convergence of \(u_j\) towards \(\hat{u}\) in \(\mathcal{U}\),
and owing to the continuous embedding also in \(L^2(\Od)\), follows
from the already established facts as follows:
\begin{equation*}
  \begin{aligned}
    0\le \limsup_{{j}\to\infty} \underline{\kappa} \|u_{j} - \hat{u}\|^2_{\cU_o^\delta}
    &\le
    \lim_{{j}\to\infty} [
    \underbrace{a_{\delta,\kappa_{j}}(u_{j},u_{j})}_{=-2I_{\delta,\kappa_j}(u_j)}
    \underbrace{-2 a_{\delta,\kappa_{j}}(u_{j},\hat{u})
    + a_{\delta,\kappa_{j}}(\hat{u},\hat{u})}_{=2I_{\delta,\kappa_{j}}(\hat{u})}]
    =
    -2I_{\delta,\hat{\kappa}}(\hat{u}) + 2I_{\delta,\hat{\kappa}}(\hat{u}) = 0.
  \end{aligned}
\end{equation*}
\end{proof}

% !TEX root = NonlocalOD.tex
%

%
\section{Existence of optimal conductivity distributions}
\label{sec:exist}

Existence of solutions for the optimization problem~\eqref{eq:min} in the special case
of compliance minimization, which corresponds to \(J(\rho,u) = \ell(u)\), has been briefly outlined
in~\cite{eb2019sensitivity}.
For the sake of keeping this manuscript self-contained, we include the short proofs
of these results here.
We also establish existence of solutions to~\eqref{eq:min} for more general objective functions,
but for a specific penalization value \(p=2\).

\subsection{Compliance minimization, convex case: \(p=1\)}
\label{subsec:convex}
Let us first consider the convex case of compliance minimization, which is obtained
by setting \(p=1\), \(J(\rho,u) = \ell(u)\) in~\eqref{eq:min}.
The argument appeals to the convexity of the problem\footnote{%
Note that the same arguments apply to the practically uninteresting case
\(0<p<1\) as well.},
see~\cite{cea1970example}.
\begin{proposition}\label{prop:compl_p1}
  The compliance minimization problem~\eqref{eq:min} (that is,
  \(J(\rho,u) = \ell(u)\)), admits an optimal solution
  \((\rho^*,u^*) \in \cD^\delta\times \cU_0^\delta\) for \(p=1\).
\end{proposition}
\begin{proof}
  In view of the existence of states for every conductivity distribution
  \(\rho \in \cD^\delta\), and also since these states satisfy the energy minimization
  principle~\eqref{eq:nl_dir_pple}, our optimal design problem can be equivalently stated as the
  following saddle point problem:
  \begin{equation*}
    \max_{\rho \in \cD^\delta} \min_{u \in \cU_0^\delta} I_{\delta,\kappa}(u),
  \end{equation*}
  where \(\kappa(x,x')= \rho^{1/2}(x)\rho^{1/2}(x')\).
  Note that the map \(\R^2_{+}\ni (\xi,\eta) \mapsto \xi^{1/2}\eta^{1/2} \in \R\)
  is concave.
  Consequently, the map \(\cD^\delta \ni \rho \mapsto I_{\delta,\kappa}(u) \in \R\)
  is concave and continuous (with respect to \(L^\infty(\Od)\)-norm) for each
  \(u\in \cU_0^\delta\), where continuity is owing to the dominated Lebesgue convergence theorem.
  Therefore, the map \(\cD^\delta \ni \rho \mapsto I_{\delta,\kappa}(u) \in \R\)
  is weak\(^*\) sequentially upper semicontinuous in \(L^\infty(\Od)\).
  This property is preserved under taking the infimum, therefore
  \(\cD^\delta \ni \rho \mapsto \min_{u\in\cU_0^\delta} I_{\delta,\kappa}(u) \in \R\)
  is weak\(^*\) sequentially upper semicontinuous in \(L^\infty(\Od)\).
  Finally, the set \(\cD^\delta\) is non-empty, closed and convex in \(L^\infty(\Od)\),
  thereby also weak\(^*\) sequentially compact.
  It only remains to apply the Weierstrass' existence theorem to conclude
  the proof.
\end{proof}
\subsection{Compliance minimization, nonconvex case: \(1<p\le 2\)}
\label{subsec:nonconvex}
Contrary to the case \(p=1\), where lower semicontinuity in the appropriate weak\(^*\)
topology holds in the local case thereby ensuring the existence of optimal solutions~\cite{cea1970example},
for \(p>1\) the nonlinear dependence on the control variable
in the state equation destroys these properties.
Consequently, additional reqularization is required to guarantee the existence of optimal
solutions for the local compliance minimization problem with SIMP, which is also
confirmed and reflected in numerous numerical algorithms based on SIMP
method~\cite{allaire2012shape,bendose2003topology}.
Surprisingly, the non-local compliance minimization problem still attains its infimum even in the
non-convex case \(1<p\le 2\).
We begin with the following simple statement.
\begin{lemma}\label{prop:wstar_cont}
  The map \(\cD^\delta \ni \rho \mapsto \rho(x)\rho(x') \in L^\infty(\Od\times\Od)\)
  is sequentially continuous in the weak\(^*\) topology of \(L^\infty\).
\end{lemma}
\begin{proof}
  Consider a sequence \(\rho_k \in \cD^\delta\), \(k=1,2,\dots\) with
  \(\rho_k \rightharpoonup \hat{\rho}\), weak\(^*\) in \(L^\infty(\Od)\).
  Let us take an arbitrary \(\psi \in L^1(\Od\times\Od)\).
  Owing to Fubini's theorem, the sequence \(\phi_k(x) = \int_{\Od} \psi(x,x')\rho_k(x')\dx'\),
  \(k=1,2,\dots\),
  converges towards \(\hat{\phi}(x) = \int_{\Od} \psi(x,x')\hat{\rho}(x')\dx'\),
  for almost all \(x \in \Od\).
  As the elements of this sequence are dominated by an \(L^1(\Od)\) function
  \(\overline{\rho}\int_{\Od} |\psi(x,x')|\dx'\),  Lebesgue's dominated convergence
  theorem applies implying that \(\lim_{k\to\infty} \|\phi_k-\hat{\phi}\|_{L^1(\Od)} = 0\).
  Finally,
  \begin{equation*}
    \begin{aligned}
      0 &\le
      \lim_{k\to\infty} \bigg|
      \int_{\Od}\int_{\Od} \psi(x,x')[\rho_k(x')\rho_k(x)-\hat{\rho}(x')\hat{\rho}(x)]\dx'\dx
      \bigg|
      =
      \lim_{k\to\infty} \bigg|\int_{\Od} [\phi_k(x)\rho_k(x) - \hat{\phi}(x)\hat{\rho}(x)]\dx \bigg|
      \\ & \le
      \lim_{k\to\infty} \bigg|\int_{\Od} \hat{\phi}(x)[\rho_k(x) - \hat{\rho}(x)]\dx \bigg|
      +
      \lim_{k\to\infty} \overline{\rho} \|\phi_k-\hat{\phi}\|_{L^1(\Od)}
      = 0.
    \end{aligned}
  \end{equation*}
\end{proof}
With this in mind, we can extend Proposition~\ref{prop:compl_p1} to
the non-convex case.
\begin{proposition}\label{prop:compl_p12}
  The compliance minimization problem~\eqref{eq:min} (that is,
  \(J(\rho,u) = \ell(u)\)), admits an optimal solution
  \((\rho^*,u^*) \in \cD^\delta\times \cU_0^\delta\) for \(1<p\le 2\).
\end{proposition}
\begin{proof}
  As in the proof of Proposition~\ref{prop:compl_p1} it is sufficient to
  establish the weak\(^*\) sequential upper semicontinuity of the map
  \(\cD^\delta \ni \rho \mapsto I_{\delta,\kappa}(u) \in \R\) with
  \(\kappa(x,x')= \rho^{p/2}(x)\rho^{p/2}(x')\), since upper semicontinuity
  is preserved under taking minimum over \(u \in \cU_0^\delta\).
  The required property follows easily from the norm-continuity and concavity of
  the map \(\cK^\delta \ni \tilde{\kappa} \mapsto I_{\delta,\tilde{\kappa}^{p/2}}(u) \in \R\)
  for \(0<p\le 2\) and Lemma~\ref{prop:wstar_cont}.
\end{proof}
\subsection{More general objective functions: \(p=2\)}
\label{subsec:nonconvex_gen}
SIMP has been successfully utilized within other contexts than compliance
minimization, see~\cite{bendose2003topology}.
The non-local optimal design problem we consider admits optimal solutions
without the need for further regularization in the special case \(p=2\)
for a wide class of objective functions.
\begin{proposition}
  Let \(p=2\), and assume that the objective function
  \((\rho,u)\mapsto J(\rho,u)\)
  is sequentially lower semicontinuous with respect to
  weak\(^*\) topology of \(L^\infty(\Od)\) \(\times\) norm topology of \(\cU_0^\delta\).
  Then the optimal design problem~\eqref{eq:min} admits an optimal solution
  \((\rho^*,u^*) \in \cD^\delta\times \cU_0^\delta\).
\end{proposition}
\begin{proof}
  In view of weak\(^*\) compactness of \(\cD^\delta\) in \(L^\infty(\Od)\),
  in order to apply the direct method of calculus of variations and conclude
  the existence of optimal solutions it is sufficient to establish that for an arbitrary
  minimizing sequence \((\rho_k,u_k) \in \cD^\delta\times\cU_0^\delta\) with
  \(\rho_k \rightharpoonup \hat{\rho}\), weak\(^*\) in \(L^\infty(\Od)\),
  we have the corresponding convergence \(u_k \to \hat{u}\) in \(\cU_0^\delta\),
  where \(\hat{u} = \cS_{\delta}(\hat{\kappa})\) with \(\hat{\kappa}(x,x') = \hat{\rho}(x)\hat{\rho}(x')\).
  However, this follows immediately from Theorem~\ref{thm:dsgn2state}
  in view of Proposition~\ref{prop:wstar_cont}.
\end{proof}

% !TEX root = NonlocalOD.tex
%

\section{Convergence to the local problem as \(\delta\to 0\): connection to SIMP}
\label{sec:gamma}
We will now turn our attention to the relationship between the non-local compliance minimization problem
and the local one, which arizes as a natural candidate for the limiting object for vanishing
non-local interaction horizons \(\delta \to 0\).
More specifically, we would like to understand whether infimum values of the nonlocal problems converge
to the infimum of the local problem, and/or whether sequences of minimizers of the nonlocal problems
converge towards minimizers of the local problem.
The standard framework for studying variational convergence of functionals,
which is equipped with the precise vocabulary for formulating and answering such questions, is that
of \(\Gamma\)-convergence~\cite{braides_beginners}.
Unfortunately, in our situation it is impossible to expect the local compliance minimization
problem to be the \(\Gamma\)-limit of the nonlocal compliance minimization problems for any \(p\in[1,2]\).
Indeed, the \(\Gamma\)-limit is always a lower semicontinuous functional in the topology in which
the \(\Gamma\)-convergence is set~\cite{braides_beginners}.
However, the local compliance functional in the presense of SIMP penalization with \(p>1\)
is \emph{not} lower semicontinuous in a relevant topology.
This is precisely the fundamental reason for the lack of optimal solutions to the local
compliance minimization problem in this situation, the fact which is well documented and understood
in the literature~\cite{bendose2003topology,allaire2012shape}.
In spite of this unfortunate insurmountable obstacle, in this section we would like to
investigate what kind of relationship between the two problems can be salvaged for any \(p\in[1,2]\).

%-------------------------------------------------------------------------------

In order to succinctly discuss convergence of minimizers it would be convenient to
put them into the same function space, which is not a priori the case given the
fact that \(\Od\) decreases to \(\O\) as \(\delta\to 0\).
Since we are only concerned with small \(\delta >0\), we fix an arbitrary
\(\delta_0 > 0\) and will only consider \(\delta \in ]0,\delta_0[\).
This allows us to consider material distributions to be elements of
the ``largest'' space \(L^\infty(\O_{\delta_0})\), extending them by
\(\underline{\rho}\) outside of their domain of definition \(\Omega_\delta\),
\(0<\delta<\delta_0\).
(The same applies to the limiting local model, if we ``by continuity'' put \(\Omega_0 = \Omega\).)
In a similar fashion we will extend the state functions by \(0\) outside of their
domain of definition \(\Omega_\delta\), \(0\le \delta<\delta_0\).

With such an extension we have that both \(\compnl\) and \(\comploc\) are defined
on a subset of the same function space, \(L^\infty(\O_{\delta_0})\), which we equip
with weak\(^*\) topology.
Convergence \(\delta\to 0\) will be understood as convergence for any sequence
\(\delta_j\to 0\) as \(j\to \infty\).

%%Throughout this section we will assume that the SIMP ``exponent'' \(p \in [1,2]\).

%-------------------------------------------------------------------------------

\subsection{``\(\Gamma\)-lower semi-continuity''}
%

%if call
%\[m=\int_{\rho\in cD} J(\rho),\quad m_{\delta}=\int_{\rho\in cD^\delta} J_\delta(\rho),\]
%then
%\[ m\le J(\rho)=\lim_{\delta\to 0} J_\delta(\rho)\le \lim_{\delta\to 0} m_\delta.\]

The first result is in the spirit of the \(\liminf\)-inequality of \(\Gamma\)-convergence,
but with the unfortunate exception that the functional arguments do not converge in the
natural topology of the function space we work with, expect for \(p=1\).

\begin{proposition}\label{prop:LI}
  Let \(\rho_\delta \in L^\infty(\O_{\delta_0})\) be such that
  \(\underline{\rho}\le \rho_\delta(x),\,\rho(x) \le \overline{\rho}\) for almost all \(x\in \O_{\delta_0}\).
  Assume that \(\rho^p_\delta\rightharpoonup \rho^p\), weak\(^*\) in \(L^\infty(\O_{\delta_0})\) as $\delta\to 0$.
  Let \(u_{\delta} = \cS_{\delta}(\kappa_\delta)\) be the solution of~\eqref{eq:w},
  where \(\kappa_\delta(x,x')=\rho_\delta^{p/2}(x) \rho_\delta^{p/2}(x')\),
  and let \(u=\cS(\rho^p)\) be the solution to~\eqref{eq:w_loc}.
  Then
  \begin{equation*}
    \liminf_{\delta\to 0} \compnl(\rho_\delta) \ge \comploc(\rho).
  \end{equation*}
\end{proposition}
\begin{proof}
  For convenience we put \(\kappa(x,x') = \rho^{p/2}(x)\rho^{p/2}(x')\);
  then \(\lim_{\delta\to 0}\kappa_\delta(x,x') =\kappa(x,x')\), for almost all
  \((x,x')\in\O_{\delta_0}^2\).
  Note that owing to~\cite[Theorem~1]{BBM}
  we have the inclusion \(u\in \cU_0^\delta\), for all \(\delta  \in ]0,\delta_0[\).
  Therefore, we can test~\eqref{eq:w} with \(v=u\) to get the equality
  \begin{equation*}
    a_{\delta,\kappa_\delta}(u_\delta,u) = \ell(u) = a_{\rho^p}(u,u).
  \end{equation*}
  Consequently, the difference \(L_\delta := \compnl(\rho_\delta)-\comploc(\rho)=
  \ell(u_\delta-u)\) can be written as
  \begin{equation*}
    \begin{aligned}
      L_\delta &= a_{\delta,\kappa_\delta}(u_\delta,u_\delta) - 2a_{\delta,\kappa_\delta}(u_\delta,u)
                + a_{\rho^p}(u,u)
        \\&=
        \underbrace{[a_{\delta,\kappa_\delta}(u_\delta,u_\delta) - 2a_{\delta,\kappa_\delta}(u_\delta,u)
          + a_{\delta,\kappa_\delta}(u,u)]}_{=:L^{(1)}_\delta}
        + \underbrace{[a_{\rho^p}(u,u)-a_{\delta,\kappa_\delta}(u,u)]}_{=:L^{(2)}_\delta}.
          %\\&
        %+ \underbrace{[a_{\delta,\kappa}(u,u)-a_{\delta,\kappa}(\phi,\phi)]}_{=:L^{(3,1)}_\delta}
        %+ \underbrace{[a_{\delta,\kappa_\delta}(\phi,\phi)-a_{\delta,\kappa_\delta}(u,u)]}_{=:L^{(3,2)}_\delta}
        %+ \underbrace{[a_{\delta,\kappa}(\phi,\phi)-a_{\delta,\kappa_\delta}(\phi,\phi)]}_{=:L^{(3,3)}_\delta},
    \end{aligned}
  \end{equation*}
%  where \(\phi \in C^\infty_c(\O)\) is arbitrary.
%  The terms on the first line have non-negative \(\liminf\), the terms on the
%  second line can be made arbitrarily small with a proper choice of \(\phi\) close to
%  \(u\), and then \(\delta\) small enough.

  Now we have \(L^{(1)}_\delta = a_{\delta,\kappa_\delta}(u_\delta-u,u_\delta-u)\ge 0\)
  because \(a_{\delta,\kappa_\delta}(\cdot,\cdot)\) is an inner product on \(\cU^\delta_0\), and therefore
  \[L_\delta\ge L_\delta^{(2)},\]
  and hence in order to prove the result it is enough to show that
   \[ \liminf_{\delta\to 0} L_\delta^{(2)}\ge 0.\]
   But, applying Young's inequality to the nonlocal conductivity
   \[ \rho_\delta^{p/2}(x) \rho_\delta^{p/2}(x') \le \frac{1}{2}\left(\rho_\delta^{p}(x)+\rho_\delta^{p}(x')  \right)\]
   we have that
   \[
\begin{aligned}
-a_{\delta,\kappa_\delta}(u,u)
&\ge
-\int_{\O_{\delta_0}}\int_{\O_{\delta_0}}
\frac{\rho^{p}_\delta(x)+\rho^p_\delta(x')}{2}
A_{\delta}(|x-x'|)\frac{(u(x)-u(x'))^2}{|x-x'|^2}\dx\dx'
\\&=
-\int_{\O_{\delta_0}}\rho^{p}_\delta(x')\int_{\O_{\delta_0}}
A_{\delta}(|x-x'|)\frac{(u(x)-u(x'))^2}{|x-x'|^2}\dx\dx',
\end{aligned}
\]
and the last term converges to \(-a_{\rho^p}(u,u)\) owing to~\cite[Corollary~1]{BBM} and the weak convergence of \(\rho^{p}_\delta\). Consequently
\[\liminf_{\delta\to 0} [a_{\rho^p}(u,u)-a_{\delta,\kappa_\delta}(u,u)] \ge 0\]
  and the proof is finished.
\end{proof}

%--------------------------------------------------------------------------------

\subsection{Pointwise convergence: \(\lim_{\delta\to 0} \compnl(\rho) = \comploc(\rho)\)}
The following result establishes the pointwise convergence of \(\compnl\) to \(\comploc\)
in \(L^\infty(\O_{\delta_0})\).
It should be understood as a \(\limsup\)-inequality in \(\Gamma\)-convergence,
where the recovering sequence is the constant sequence.

In order to prove it we need the following lemma,
which establishes that the estimate~\eqref{eq:L2coerc} can be made uniform with
respect to small \(\delta > 0\).
\begin{lemma}\label{poincare}
  There exists a constant \(\hat{\delta} \in ]0,\delta_0[\), and a constant
  \(C_{\hat{\delta}}>0\), independent from \(\delta\) and \(u\), such that we have the
  inequality \(C_{\hat{\delta}}\|u\|_{L^2(\O)}\le \|u\|_{\cU^\delta_0}\)
  for any \(\delta\in ]0,\hat{\delta}[\) and \(u\in \cU_0^\delta\).
\end{lemma}
\begin{proof}
For the sake of contradiction, we assume that for each \(j=1,2,\dots\) there is
\(\delta_j\in ]0,\delta_0[\), and \(u_j\in \cU_0^{\delta_j}\) such that
\begin{equation*}
  \lim_{j\to\infty} \delta_j = 0, \qquad
  \lim_{j\to\infty} \|u_j\|_{\cU^{\delta_j}_0} = 0, \quad\text{but}\quad
  \|u_j\|_{L^2(\O)}=1.
\end{equation*}
Note that since \(u_j\) is extended by zero outside of \(\O_{\delta_j}\),
\(u_j \equiv 0\) in \(\O_{\delta_0}\setminus \O\).
When combined with the smallness of support of \(A_{\delta_j}(|\cdot|)\), see~\eqref{eq:smallsupport},
this implies the equality
\begin{equation*}
  \int_{\O_{\delta_0}}\int_{\O_{\delta_0}} A_{\delta_j}(|x-x'|) \frac{(u_j(x)-u_j(x'))^2}{|x-x'|^2}\dx'\dx
  =a_{\delta_j,1}(u_j,u_j)=\|u_j\|^2_{\cU^{\delta_j}_0}.
\end{equation*}
We can therefore apply~\cite[Theorem~1.2]{ponce2004estimate}, which asserts that
the sequence \(u_j\) is relatively compact in \(L^2(\O_{\delta_0})\) with all
its accumulation points being in \(W^{1,2}(\O_{\delta_0})\).
Let \(u_0 \in W^{1,2}(\Omega_{\delta_0})\) be such an \(L^2(\O_{\delta_0})\)-accumulation point
of \(u_j\); in particular \(\|u_0\|_{L^2(\O)}=1\).
Since the accumulation point does not depend on any finite number of terms in the
sequence \(u_j\), we can utilize the estimate in~\cite[Theorem~1.2]{ponce2004estimate}
as follows:
\begin{equation*}
  \int_{\O_{\delta_0}}|\nabla u_0(x)|^2\dx\le \limsup_{j \to \infty} \{a_{\delta_j,1}(u_j,u_j)\}=0.
\end{equation*}
Therefore, \(u_0\) must be a constant on \(\O_{\delta_0}\).
On the other hand we have the pointwise (in fact, finite) convergence
\(\lim_{j\to\infty} u_j(x) = 0\), \(x \in \O_{\delta_0} \setminus \cl\O\),
and consequently \(u_0\equiv 0\) on \(\O_{\delta_0} \setminus \cl\O\).
Therefore \(u_0\equiv 0\) on \(\Omega_{\delta_0}\), which contradicts the
previously established fact \(\|u_0\|_{L^2(\O)} = 1\).
\end{proof}
\begin{remark}
  Note that Lemma~\ref{poincare} implies the coercivity of the
  nonlocal equation~\eqref{eq:w} even in the absence of assumption~\eqref{fractionallb},
  as~\cite[Theorem~1.2]{ponce2004estimate} does not require such a condition.
  However, as we mentioned previously, for our purposes the assumption~\eqref{fractionallb} is a natural hypothesis,
  as it implies the continuous embedding of \(\mathcal{U}_0^\delta\) into \(W^{s,2}(\Od)\)
  and therefore also the compact embedding into \(L^2(\Od)\).
\end{remark}

Before staying the main result of this section we need to define the following class of designs:
\begin{equation*}
\tilde{\cD}^\delta=\bigg\{\,
\rho \in \cD^\delta\;:\; \rho(x)=\sum_{k=1}^K \alpha_k\chi_{B_k}(x),\, K\in\mathbb{N},\, \alpha_k>0,\, B_k \text{\ open, pairwise disjoint, and such that\ }
\overline{\cup_{k=1}^K B_k}\supset \O_{\delta}\,\bigg\}.
\end{equation*}
Note that \(\tilde{\cD}^\delta\) is not just the class of simple functions in \(\cD^\delta\), but the class on simple functions \emph{supported} on open sets.  This subtle but important restriction is going to be needed to apply \(\Gamma\)-convergence results of~\cite{ponce2004gamma,bellido2015} in the Step~2 of the proof below.

\begin{proposition}\label{prop:limit}
Consider an arbitrary \(\rho\in \tilde{\cD}^\delta\).
Let \(u_\delta = \cS_\delta(\kappa) \in \cU_0^\delta\) be the sequence of solutions to~\eqref{eq:w}
corresponding to a fixed \(f \in L^2(\O)\) and \(\kappa(x,x')= \rho^{p/2}(x)\rho^{p/2}(x')\),
but varying \(\delta \to 0\).
Let further \(u = \cS^{\text{loc}}(\rho^{p}) \in W^{1,2}_0(\O)\) be the weak solution to
the local generalized Laplace problem~\eqref{eq:w_loc}.

Then
\begin{equation*}
  \lim_{\delta\to 0} \|u_\delta - u\|_{L^2(\O)} = 0, \text{\ and\ }
  \lim_{\delta\to 0} a_{\delta,\kappa}(u_\delta,u_\delta) =
  a_{\rho^p}(u,u).
\end{equation*}
Consequently,
\begin{equation*}
 \lim_{\delta\to 0} \compnl(\rho)=  \comploc(\rho),
\end{equation*}
for any \(\rho \in \cD\).
\end{proposition}
\begin{proof}
  Let \(\hat{\delta} \in ]0,\delta_0[\) and \(C_{\hat{\delta}}>0\) be those established
  in Lemma~\ref{poincare}.
  Since \(u_\delta = \cS_\delta(\kappa)\) solves~\eqref{eq:w}, for all \(\delta \in ]0,\hat{\delta}[\)
  we get the estimate
  \begin{equation*}
    C_{\hat{\delta}}\underline{\kappa} \| u_\delta\|^2_{L^2(\O)}
    \le  \underline{\kappa} a_{\delta,1}(u_\delta,u_\delta)
    \le  a_{\delta,\kappa}(u_\delta,u_\delta)
    \le \|f\|_{L^2(\O)}\|u_\delta\|_{L^2(\O)},
  \end{equation*}
  and consequently the uniform stability estimates
  \begin{equation*}
    \begin{aligned}
      \| u_\delta\|_{L^2(\O_{\hat{\delta}})} =
      \| u_\delta\|_{L^2(\O)}\le\frac{1}{C_{\hat{\delta}}\underline{\kappa}} \|f\|_{L^2(\O)},\quad\text{and}\quad
      \| u_\delta\|_{\cU^\delta_0} \le \frac{1}{C_{\hat{\delta}}^{1/2}\underline{\kappa}}\|f\|_{L^2(\O)}.
    \end{aligned}
  \end{equation*}
  Utilizing~\cite[Theorem~1.2]{ponce2004estimate} as in Lemma~\ref{poincare},
  we establish the existence of an \(L^2(\O_{\hat{\delta}})\)-accumulation point
  \(u_0 \in W^{1,2}_0(\O)\) of \(u_\delta\).
  Let \(\delta_j\), \(j=1,2,\dots\) be a sequence realizing convergence towards
  this accumulation point, that is
  \(\lim_{j\to\infty} \delta_j = 0\) and \(\lim_{j\to\infty} \|u_{\delta_j}-u_0\|_{L^2(\O)}=0\).
  As mentioned previously, in the following discussion both \(u_0\) and \(u\),
  the solution of the limiting local problem, are extended by \(0\) outside of \(\O\).

  \noindent\textbf{Step 1.} We claim that
  \begin{equation}\label{step1}
    \limsup_{j\to \infty}I_{\delta_j,\kappa}(u_{\delta_j})\le I_{\rho^p}(u).
  \end{equation}

  Indeed, owing to~\eqref{eq:kernelnrm} and~\cite[Theorem~1]{BBM} we have the bound
  \(\|u\|_{\cU^{\delta_j}_0} \le C \|u\|_{W^{1,2}(\O)}\) and therefore also the
  inclusion \(u\in \cU^{\delta_j}_0\).
  Consequently, owing to the variational characterization~\eqref{eq:nl_dir_pple}
  we have the inequalities \(I_{\delta_j,\kappa}(u_{\delta_j})\le I_{\delta_j,\kappa}(u)\),
  for each \(\delta_j\), \(j=1,2,\dots\).
  Therefore, in order to establish~\eqref{step1} it is sufficient to show the inequality
  \(\limsup_{j\to \infty} a_{\delta_j,\kappa}(u,u) \le a_{\rho^p}(u,u)\).
  But this inequality follows from the direct application of Proposition~\ref{prop:LI}
  to the constant sequence \(\rho\):
  \[\liminf_{j\to \infty} \left(-2I_{\delta_j,\kappa}(u_{\delta_j})\right)\ge \left(-2 I_{\rho^p}(u)\right)\]
  and therefore \eqref{step1} holds.

  \noindent\textbf{Step 2.} We claim that
  \begin{equation*}
    I_{\rho^p}(u_0) \le \liminf_{j\to\infty} I_{\delta_j, \kappa}(u_{\delta_j}).
  \end{equation*}
  Note that owing to the strong \(L^2\) convergence \(\lim_{j\to\infty} \|u_{\delta_j}-u_0\|_{L^2(\O)}=0\)
  it is sufficient to prove the inequality
  \begin{equation*}
    a_{\rho^p}(u_0,u_0) \le \liminf_{j\to\infty} a_{\delta_j, \kappa}(u_{\delta_j},u_{\delta_j}).
  \end{equation*}

  Let us recall that \(\rho\in\tilde{\cD}^{\delta_0}\) is a simple function supported on open sets,
  that is
  \[\rho(x)=\sum_{i=1}^I \alpha_i \chi_{B_i}(x),\]
  where \(I\in\N\), \(\alpha_i>0\) and the \(B_i\) open and pairwise disjoint such that \(\cl(\cup_{i=1}^I B_i) \supset \O_{\delta_0}\).
  Then
  \begin{equation*}
    \begin{aligned}
      a_{\delta_j,\kappa}(u_{\delta_j},u_{\delta_j}) &=
      \int_{\O_{\delta_0}} \int_{\O_{\delta_0}} \rho^{p/2}(x) \rho^{p/2}(x')
      A_{\delta_j}(|x-x'|)\frac{(u_{\delta_j}(x)-u_{\delta_j}(x'))^2}{|x-x'|^{2}}\dx'\dx
      \\ &=
      \sum_{i,j=1}^I \alpha_i^{p/2}\alpha_j^{p/2} \int_{B_i}\int_{B_j}
      A_{\delta_j}(|x-x'|)\frac{(u_{\delta_j}(x)-u_{\delta_j}(x'))^2}{|x-x'|^{2}}\dx'\dx
      \\ &\ge
      \sum_{i=1}^I \alpha_i^p \int_{B_i}\int_{B_i}
      A_{\delta_j}(|x-x'|)\frac{(u_{\delta_j}(x)-u_{\delta_j}(x'))^2}{|x-x'|^{2}}\dx'\dx.
    \end{aligned}
  \end{equation*}
  Applying the \(\Gamma\)-convergence results in \cite{ponce2004gamma,bellido2015} we conclude that
  \begin{equation}\label{eq:lbu0}
    \liminf_{j\to\infty} \int_{B_i}\int_{B_i}
    A_{\delta_j}(|x-x'|)\frac{(u_{\delta_j}(x)-u_{\delta_j}(x'))^2}{|x-x'|^{2}}\dx'\dx
    \ge \int_{B_i}|\nabla u_0(x)|^2\dx.
  \end{equation}
  Summing up these inequalities and recalling that \(u_0\) vanishes outside \(\O\)
  we can conclude that
  \begin{equation*}
    \liminf_{j\to\infty} a_{\delta_j,\kappa}(u_{\delta_j},u_{\delta_j})\ge
    a_{\rho^p}(u_0,u_0).
  \end{equation*}
  Since \(\lim_{j\to\infty} \|u_{\delta_j}-u_0\|_{L^2(\O)}=0\) this is sufficient
  to conclude the proof of step~2 for simple functions \(\rho\). Notice that the requirement of \(B_i\) to be open is necessary in order to apply the results in \cite{ponce2004gamma,bellido2015}.

  \noindent\textbf{Step 3: Conclusion.}
  Combining the inequalities obtained in steps~1 and~2 we obtain the
  following string of inequalities:
  \begin{equation*}
    \limsup_{j\to\infty} \underbrace{I_{\delta_j,\kappa}(u_{\delta_j})}_{=-0.5\comploc_{\delta_j}(\rho)}
    \le
    I_{\rho^p}(u)
    \le
    \underbrace{I_{\rho^p}(u_0)}_{=-0.5\comploc(\rho)}
    \le
    \liminf_{j\to\infty} \underbrace{I_{\delta_j,\kappa}(u_{\delta_j})}_{=-0.5\comploc_{\delta_j}(\rho)}.
  \end{equation*}
  The variational characterization of the local problem~\eqref{eq:loc_dir_pple}
  and the uniqueness of solutions to~\eqref{eq:w_loc} implies that \(u=u_0\).
  Therefore the family of solutions \(\{u_{\delta}\}_{0<\delta<\hat{\delta}}\),
  which is relatively compact in \(L^2(\O_{\hat{\delta}})\), has only one accumulation
  point, and the sequence \(\delta_j\), \(j=1,2,\dots\) selected in the beginning
  of the proof is in fact arbitrary.
  This finishes the proof of the proposition.
\end{proof}

The conclusions of Proposition~\ref{prop:limit} hold in fact for an even larger
class of material distributions than \(\tilde{\cD}^{\delta_0}\).

\begin{corollary}\label{cor:limit}
  Suppose that \(\rho\in \cD^{\delta_0}\) be such that there exists a
  sequence \(\{\rho_i\}_i\subset \tilde{\cD}^{\delta_0}\) such that
  \(\rho_i(x) \leq \rho(x)\) and
  \(\lim_{i\to \infty }\rho_i(x)=\rho(x)\) a.e. in \(x\in\O_{\delta_0}\).
  Then the conclusion of Proposition~\ref{prop:limit} holds for this \(\rho\).
\end{corollary}
\begin{proof}
  Note that Steps~1 and~3 of the proof above do not utilize the simple structure of
  \(\rho\).
  Thus, it is only necessary to amend Step~2 of the proof, which we do here.

  Owing to our assumptions, there exists a sequence of non-negative simple functions
  \(\rho_{i}\), which approximates \(\rho\) almost everywhere in \(\O_{\delta_0}\)
  from below.
  Let \(\kappa_{i}(x,x')=\rho_{i}^{p/2}(x)\rho_{i}^{p/2}(x')\).
  Then \(\kappa_i(x,x') \leq \kappa(x,x')\) for almost all \((x,x')\in \O_{\delta_0}^2\).
  Therefore, in view of Proposition~\ref{prop:limit}, for each \(i=1,2,\dots\) we can write:
  \begin{equation*}
    \liminf_{j\to\infty} a_{\delta_j,\kappa}(u_{\delta_j},u_{\delta_j})
    \ge
    \liminf_{j\to\infty} a_{\delta_j,\kappa_i}(u_{\delta_j},u_{\delta_j})
    \ge
    a_{\rho_i^p}(u_0,u_0).
  \end{equation*}
  It remains to take the limit with respect to \(i\to\infty\) and utilize
  the dominated Lebesgue convergence theorem to reach the inequality claimed
  in Step~2 of the proof of Proposition~\ref{prop:limit}.
\end{proof}

\begin{remark}
  Note that one can approximate an arbitrary bounded and measurable \(\rho\) pointwise
  with a non-decreasing sequence of \emph{simple} functions, see~\cite[Theorem~1.17]{rudin}.
  Unfortunately, in Corollary~\ref{cor:limit} we need that these simple functions are
  additionally \emph{supported} on open, and not just measurable, sets \(B_i\).
  Therefore we have not succeeded in showing the result for a general
  \(\rho\in \cD^{\delta_0}\).
  We actually believe, and conjecture, that Corollary~\ref{cor:limit} is true
  without this additional restriction.
  \end{remark}

\subsection{Discussion of \(\Gamma\)-convergence}

General \(\Gamma\)-convergence theory establishes that \(\Gamma\)-convergence together with equi-coercivity implies both convergence of infima to the minimum of the limit problem, and that any cluster point of any sequence of minimizers is a minimizer of the limit problem \cite{braides_beginners}. As we have pointed out above, the unfortunate fact that convergence of functional arguments for liminf result, Proposition \ref{prop:LI}, is not weak\(^*\) convergence in \(L^\infty(\O_{\delta_0})\) spoils the possibility of a general \(\Gamma\)-convergence result for any \(p\in]1,2]\). Furthermore, local compliance for \(p\in ]1,2]\) is not lower semicontinous and consequently cannot be a \(\Gamma\)-limit. In addition, the restricting hypothesis on the admissible material distributions in the statement of Corollary~\ref{cor:limit} rules out the possibility of obtaining a general \(\Gamma\)-convergence result  even for \(p=1\).
Still, we can point out at a few inequalities characterizing the relationship
between the non-local and the local optimal control problems in the convex case \(p=1\).

On the one hand, for each \(\delta>0\) we can take a minimizer \(\rho_\delta\)
for the non-local compliance minimization.
Clearly the sequence \(\rho_\delta\) is bounded in \(L^\infty(\O_{\delta_0})\).
Therefore, there exists a sequence \(\delta_j\), with \(\delta_j\to 0\), and
\(\rho\in \cD\), such that \(\rho_{\delta_j}\rightharpoonup \rho\) weak\(^*\) in \(L^\infty(\O_{\delta_0})\).
Taking into account that \(p=1\), Proposition~\ref{prop:LI} implies
\[m\le \comploc(\rho)\le \liminf_{j\to \infty} \mathfrak{c}_{\delta_j}(\rho_{\delta_j}) =  \liminf_{j\to \infty} m_{\delta_j},\]
where
\[m=\inf_{\rho\in \cD} \comploc(\rho), \quad m_\delta =\min_{\rho \in \cD^\delta} \mathfrak{c}_{\delta}(\rho).\]
As this argument can be made for any sequence $\delta_j$ converging to zero, we have that
\begin{equation}\label{eq:infimaineq} m\le \liminf_{\delta\to 0} m_\delta.\end{equation}

Furthermore, if \(\inf_{\rho\in \cD} \comploc(\rho)\) is attained at \(\hat{\rho}\in \cD\)
satisfying the assumptions of Corollary~\ref{cor:limit}, then the inequality~\eqref{eq:infimaineq} becomes equality:
\begin{equation*}
  m= \comploc(\hat{\rho}) = \lim_{j\to \infty} \mathfrak{c}_{\delta_j}(\hat{\rho})
  \geq \limsup_{j\to \infty} m_{\delta_j},
\end{equation*}
for any \(\delta_j\to 0\).

Finally, if Corollary~\ref{cor:limit} were true for any \(\rho\in \cD^{\delta_0}\),
then we would have that \(\compnl\) \(\Gamma\)-converges to \(\comploc\) as \(\delta\) goes to 0.
Indeed, weak\(^*\) topology is metrizable on bounded sets of \(L^\infty(\O_{\delta_0})\)
(cf.~\cite[Theorem~3.28]{brezis}).
Therefore \(\mathcal{A}^{\delta_0}\) equipped with the weak\(^*\) topology is a metric space, and therefore \(\Gamma\)-convergence requires only two facts to hold, namely, limsup and liminf inequalities.
These would then be direct consequences of Proposition~\ref{prop:LI} and Corollary~\ref{cor:limit},
respectively.

% !TEX root = NonlocalOD.tex
%
\section{Numerical experiments}
\label{sec:num}
The objective of this section is to numerically illustrate the behaviour
of the proposed optimization model, with emphasis on the results
established in the  previous sections.
All our numerical experiments are performed with \(n=2\),
\(\O=]0,1[^n\), \(A_\delta(|x|) = c_{\text{nrm}}|x|^{-(n+2s-2)}\max\{0,\delta^2-|x|^2\}^\beta\),
with \(\beta = 3.0\) and \(c_{\text{nrm}}\) determined from~\eqref{eq:kernelnrm}.
Additionally, we use \(\underline{\rho} = 10^{-3}\), \(\overline{\rho} = 1.0\),
\(s \in \{1/3,2/3\}\), \(\delta \in \{0.05,0.1,0.2\}\), and \(p \in \{1,2\}\).
\subsection{Galerkin FEM discretization of the state equations}
The variational formulation~\eqref{eq:w} with a symmetric and coercive bilinear
form \(a_{\delta,\kappa}(\cdot,\cdot)\) naturally lends itself for an application
of Galerkin method.
In our numerical experiments we only consider polyhedral sets \(\O\),
and therefore we proceed in the standard fashion by decomposing
 \(\Od\) into a union of shape-regular simplices \(\Od^h\),
where \(h>0\) will denote a characteristic size (diameter) of the elements in
our mesh.
We make sure that \(\Od^h\) conforms with the subdivision of
\(\Od\) into \(\O\) and \(\Gnl\), see Fig.~\ref{fig:mesh}.
\begin{figure}
  \centering
  \includegraphics[width=0.32\columnwidth]{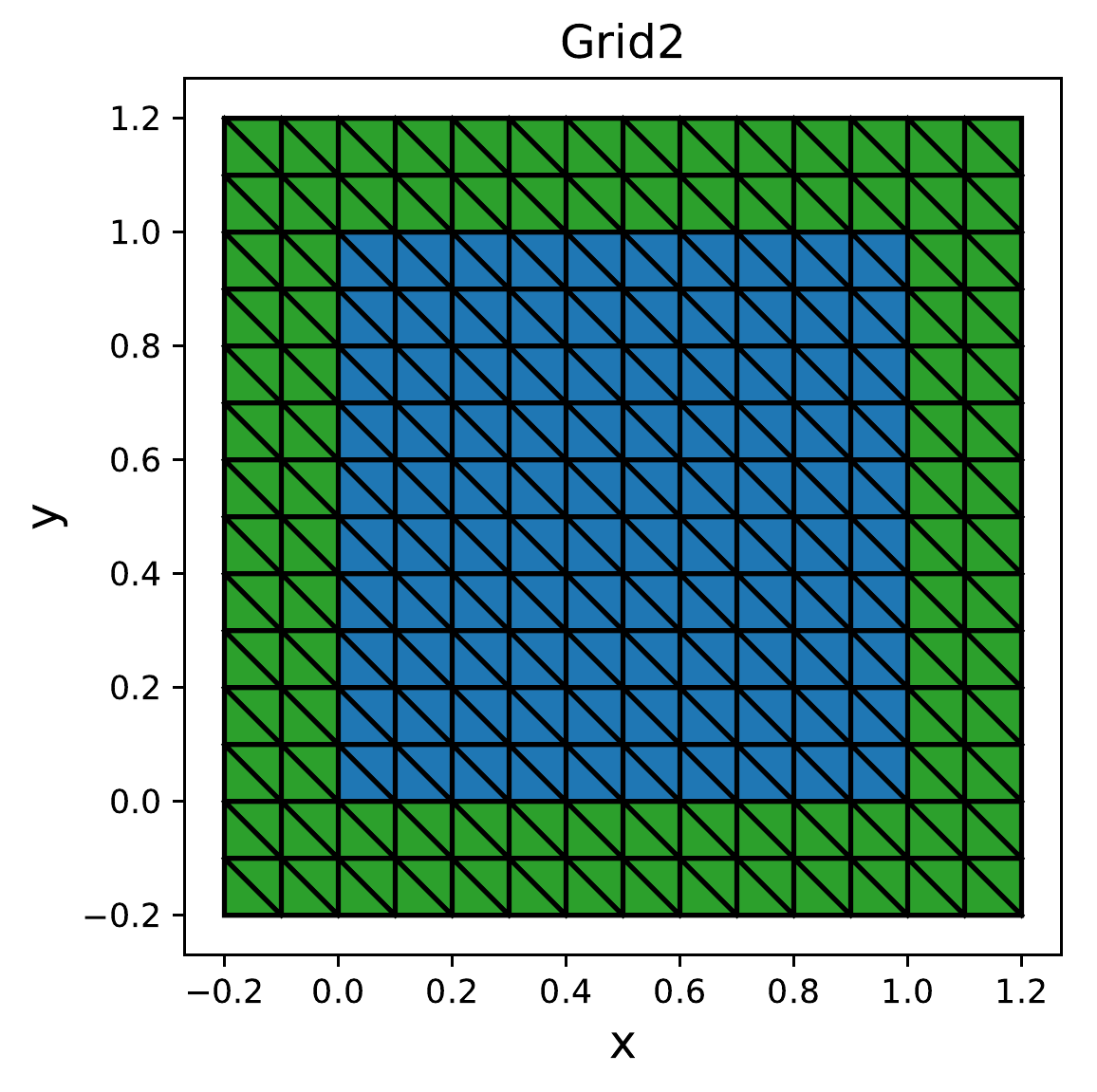}
  \includegraphics[width=0.32\columnwidth]{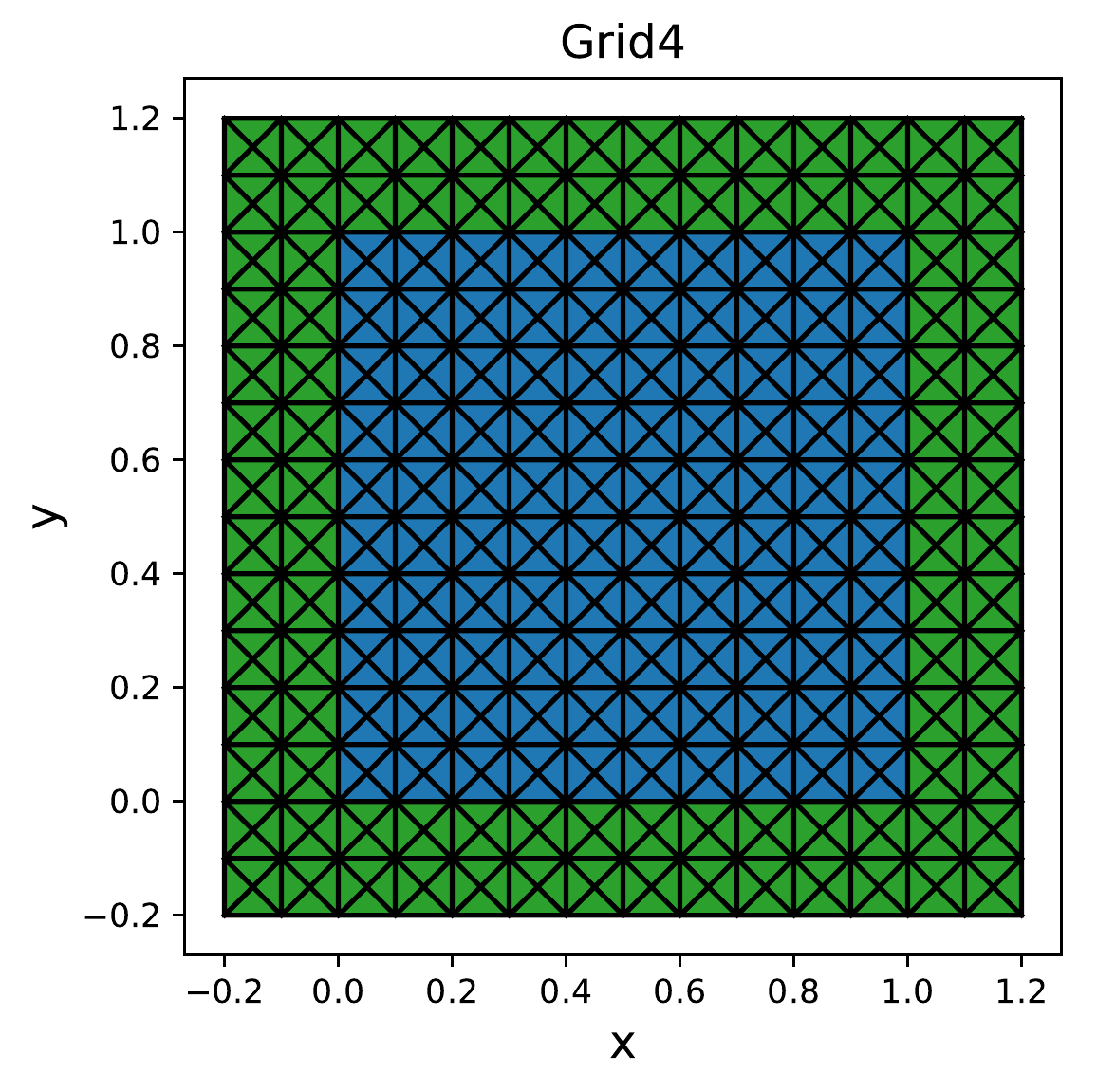}
  \includegraphics[width=0.32\columnwidth]{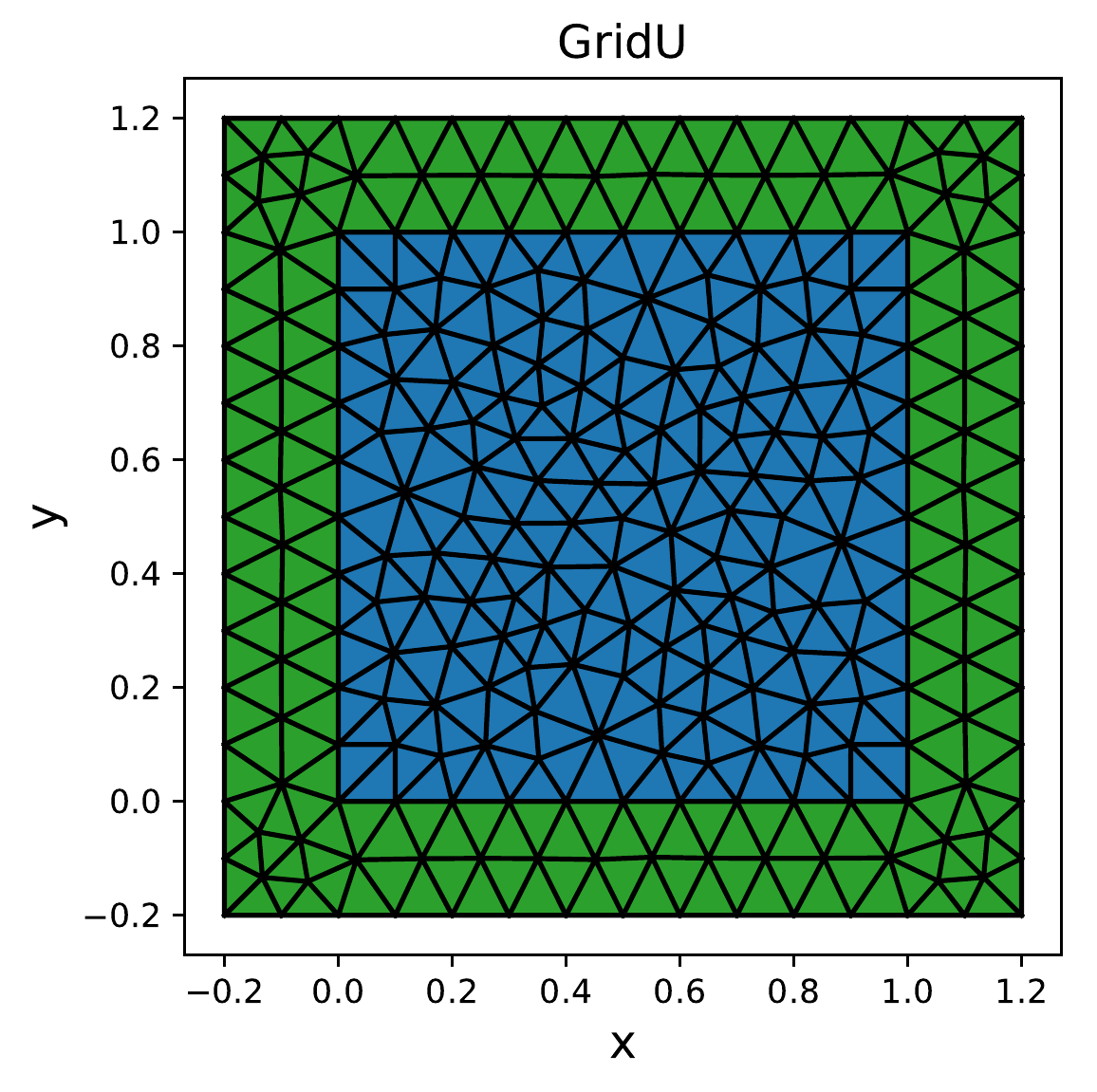}
  \caption{%
  Three types of meshes used in the present work and their corresponding labels.
  In this figure, we use \(\delta = 0.2\), and decomposition of
  \(\Od\) into \(\O\) and \(\Gnl\) is illustrated with color.
  Gmsh is utilized for generating unstructured meshes~\cite{geuzaine2009gmsh}.%
  }
  \label{fig:mesh}
\end{figure}
\begin{figure}
  \centering
  \includegraphics[width=0.35\columnwidth]{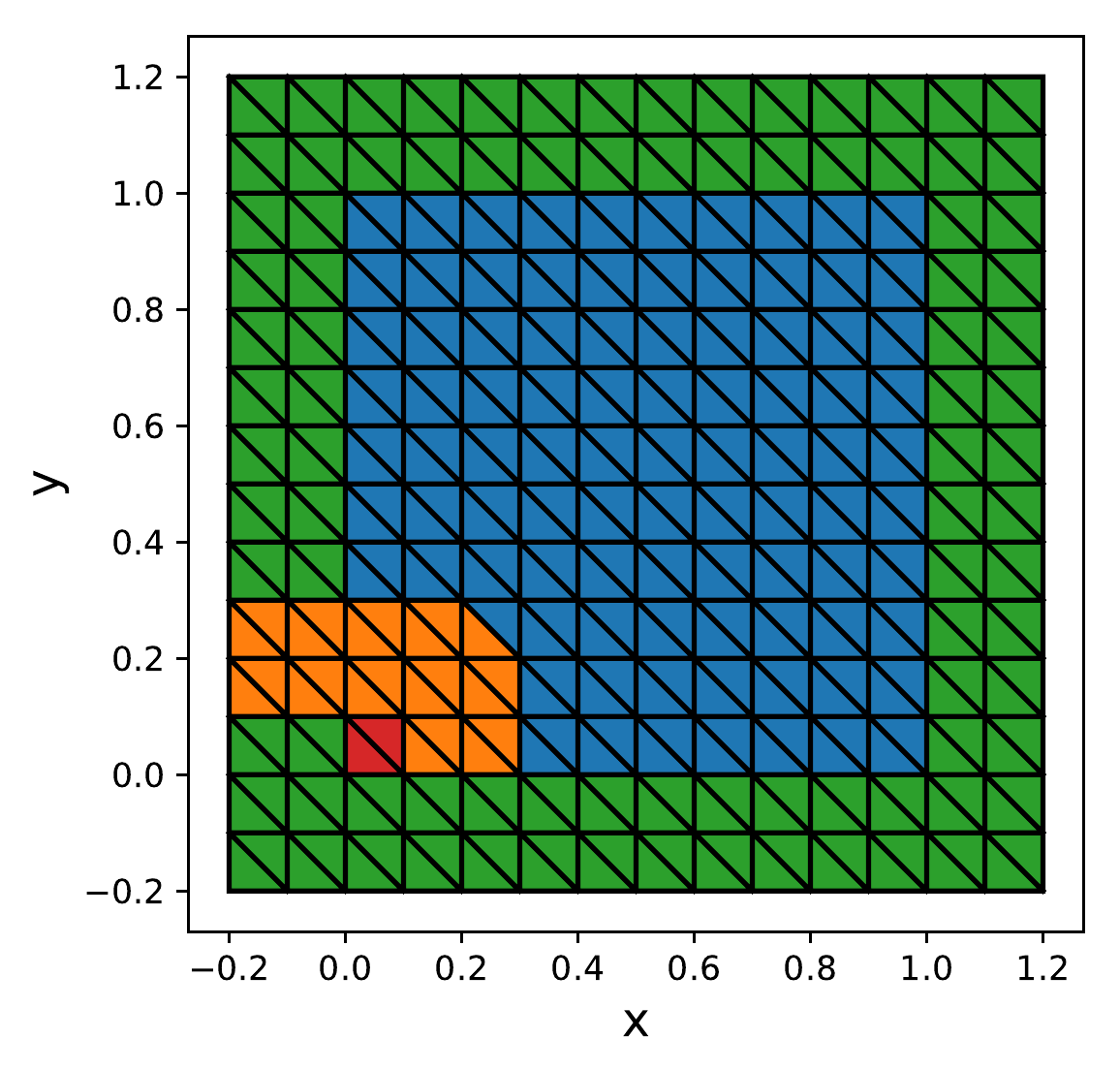}
  \includegraphics[width=0.41\columnwidth]{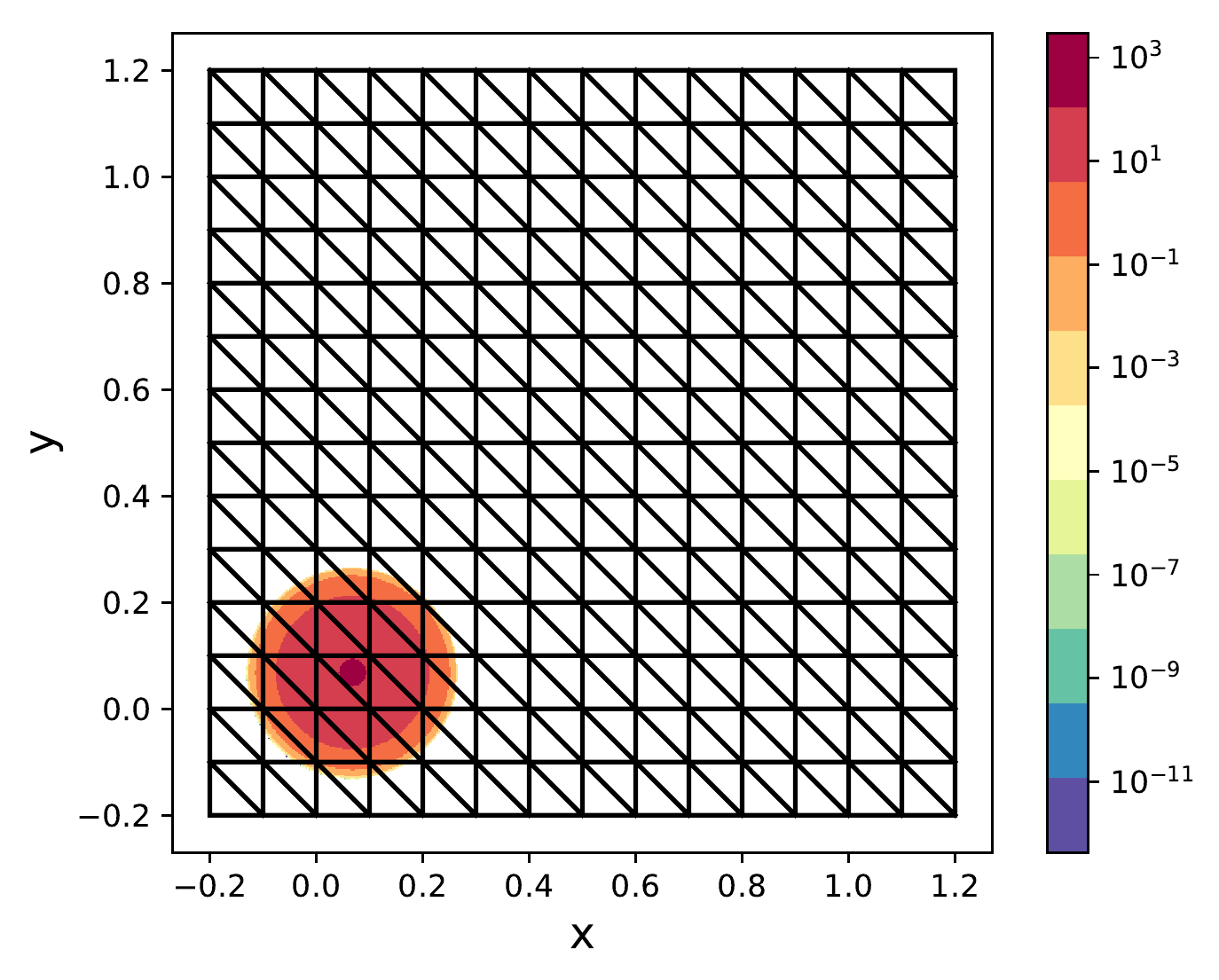}
  \caption{%
  Left: reference elements \(T_1\) (red), and elements \(T_2\) (orange), participating
  in the pre-computation of the integrals.
  Note that only half of the elements in the proximity of the reference cell need
  to be considered owing to the symmetry of the bilinear form with respect to the
  change of the integration variables \(x \leftrightarrow x'\).
  Right: contour plot of \(A_{\delta}(|\cdot-x|)\) for \(x\) located at the
  barycenter of one of the reference elements.
  \(\delta = 0.2\) and \(s=1/3\) is used.
  }
  \label{fig:mesh2}
\end{figure}
Functions in \(\cU\) will be approximated with continuous piecewise-linear
polynomials \(u^h \in \cU^h = \{\, v^h \in \cU \cap C^0(\bar{\O}_\delta) :
v^h|_T \text{\ is a linear polynomial}, \forall T \in \Od^h\,\}\).
Naturally we put \(\cU^h_0 = \cU^h \cap \cU_0\), which leads us to the following
discrete variational principle (system of linear algebraic equations): find
\(u_h \in \cU^h_0\), such that
\begin{equation}\label{eq:wh}
  a_{\delta,\kappa}(u_h,v_h) = \ell(v_h), \quad \forall v_h \in \cU_0^h.
\end{equation}

Assembly process for the right hand side of this system is completely standard,
whereas in order to assemble the left hand side of this system we need to loop
over all \emph{pairs} \((T_1,T_2)\in \Od^h\times \Od^h\) of elements in the mesh,
which are not further than the distance of \(2\delta\) from each other, and
compute the local integral contribution to \( a_{\delta,\kappa}(u_h,v_h)\),
that is, the integral over \(T_1 \times T_2\).
Note that when \(\bar{T}_1 \cap \bar{T}_2\neq\emptyset\), the integrand is
unbounded; even when \(\bar{T}_1 \cap \bar{T}_2=\emptyset\) the integrand is not
a polynomial function.
In our implementation we utilize the quadratures described in~\cite{chernov2015quadrature},
which are taylored for a nearly identical situation.\footnote{%
It should be noted that our integrands do not always satisfy the assumptions imposed
in~\cite{chernov2015quadrature} as the terms \((u(x)-u(x'))/|x-x'|\) and
\((v(x)-v(x'))/|x-x'|\) are only bounded and not continuous across the
boundaries \(\bar{T}_1 \cap \bar{T}_2\).}
In order to avoid commiting a variational crime by not integrating the bilinear
form precisely, we first estimate how many quadrature points we need for the
accurate integration; the results are reported in Fig.~\ref{fig:integration}.
Despite the fact that the assumptions imposed in~\cite{chernov2015quadrature}
are not always satisfied we observe exponential convergence of the quadratures.
However, note the unusial scaling of the \(x\)-axis; in the most singular case
\(k=2\) corresponding to \(T_1=T_2\) we need approximately \(15^5 = 759375\)
quadrature points (when using Gauss--Jacobi quadrature in the singular direction,
see~\cite{chernov2015quadrature} for details) to achieve nearly full IEEE double
precision accuracy before the round off errors start to play a role!
\begin{figure}
  \centering
  \begin{tabular}{cc}
    \includegraphics[width=0.48\columnwidth]{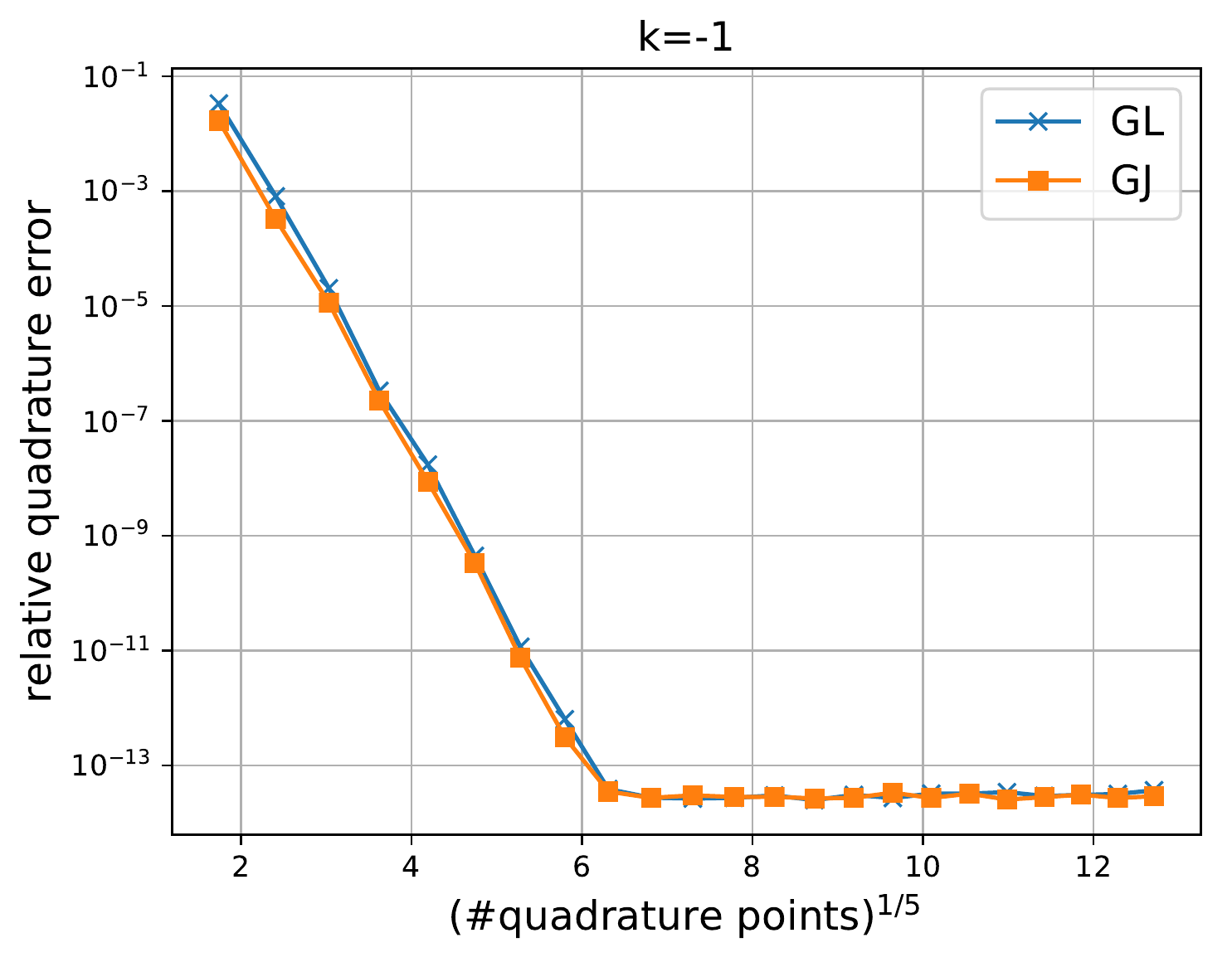}&
    \includegraphics[width=0.48\columnwidth]{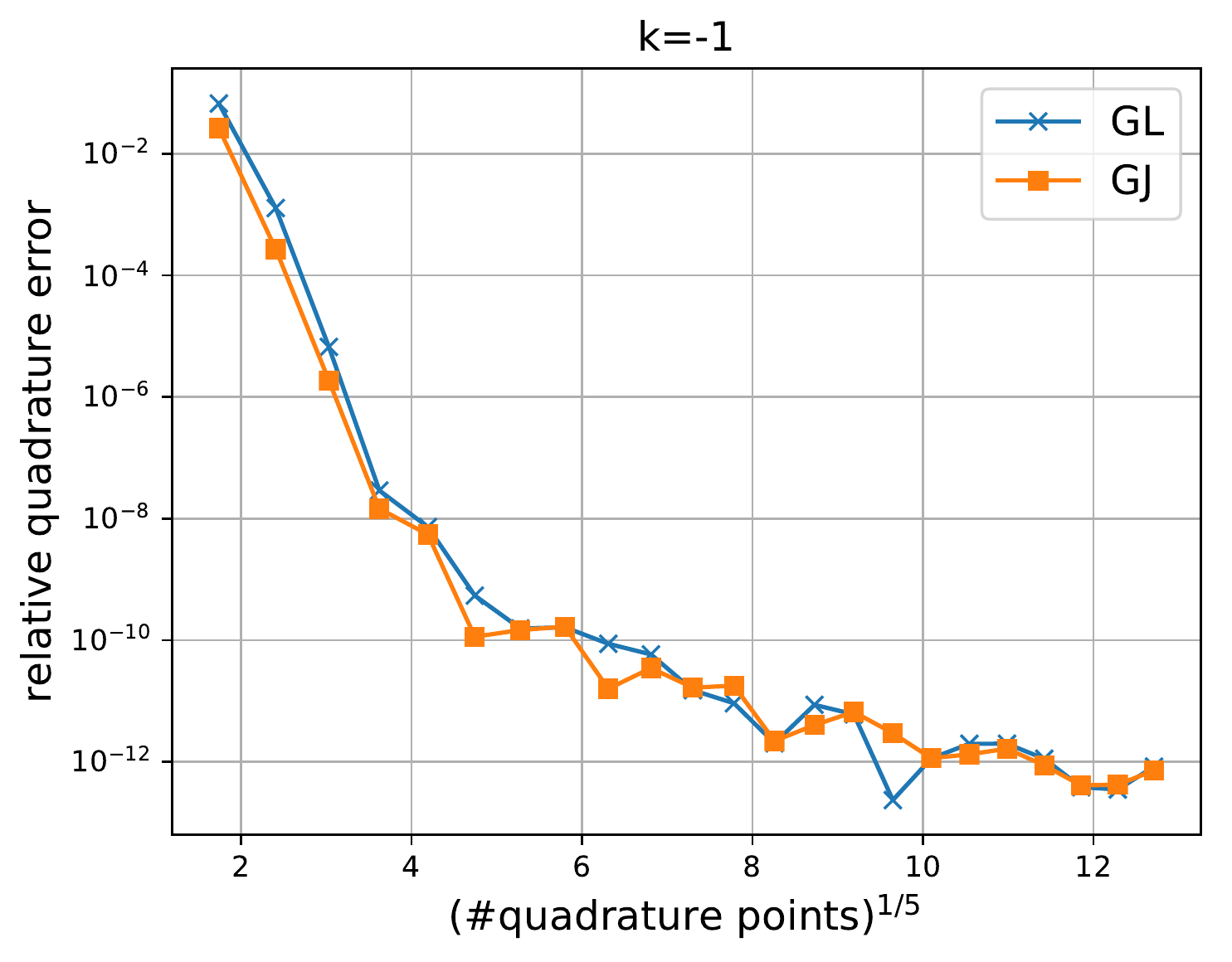}\\
    (a) & (b)\\
    \includegraphics[width=0.48\columnwidth]{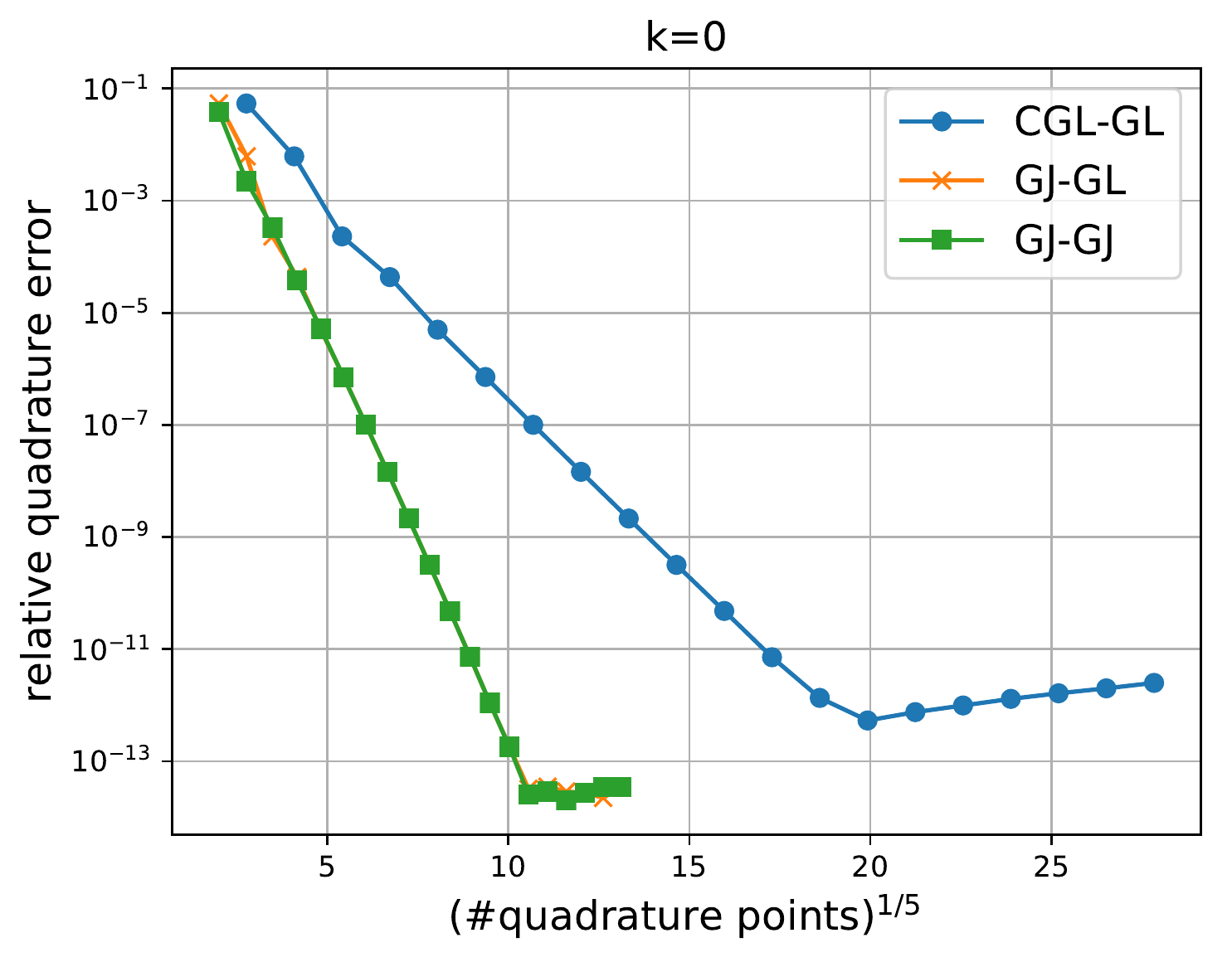}&
    \includegraphics[width=0.48\columnwidth]{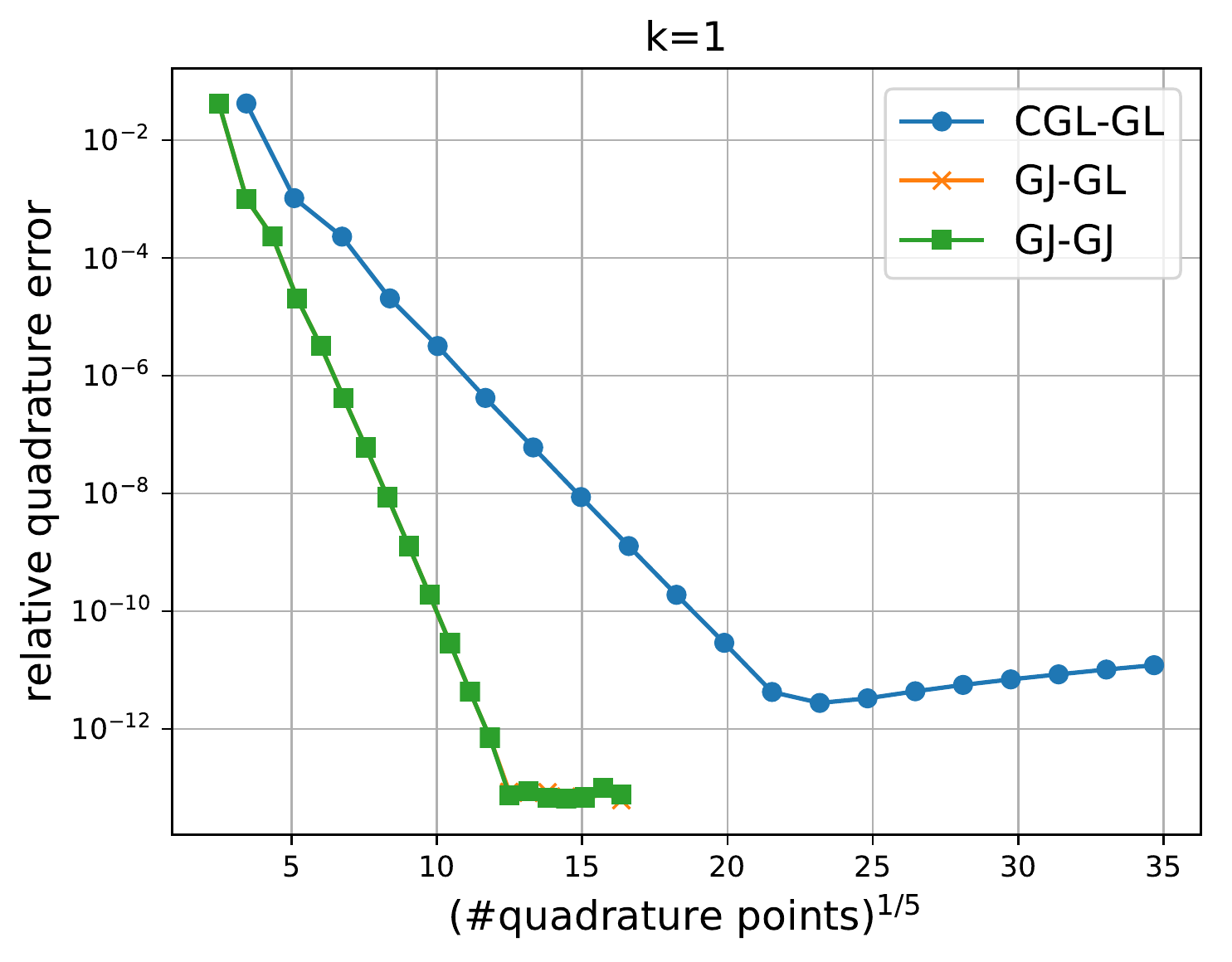}\\
    (c) & (d)\\
    \includegraphics[width=0.48\columnwidth]{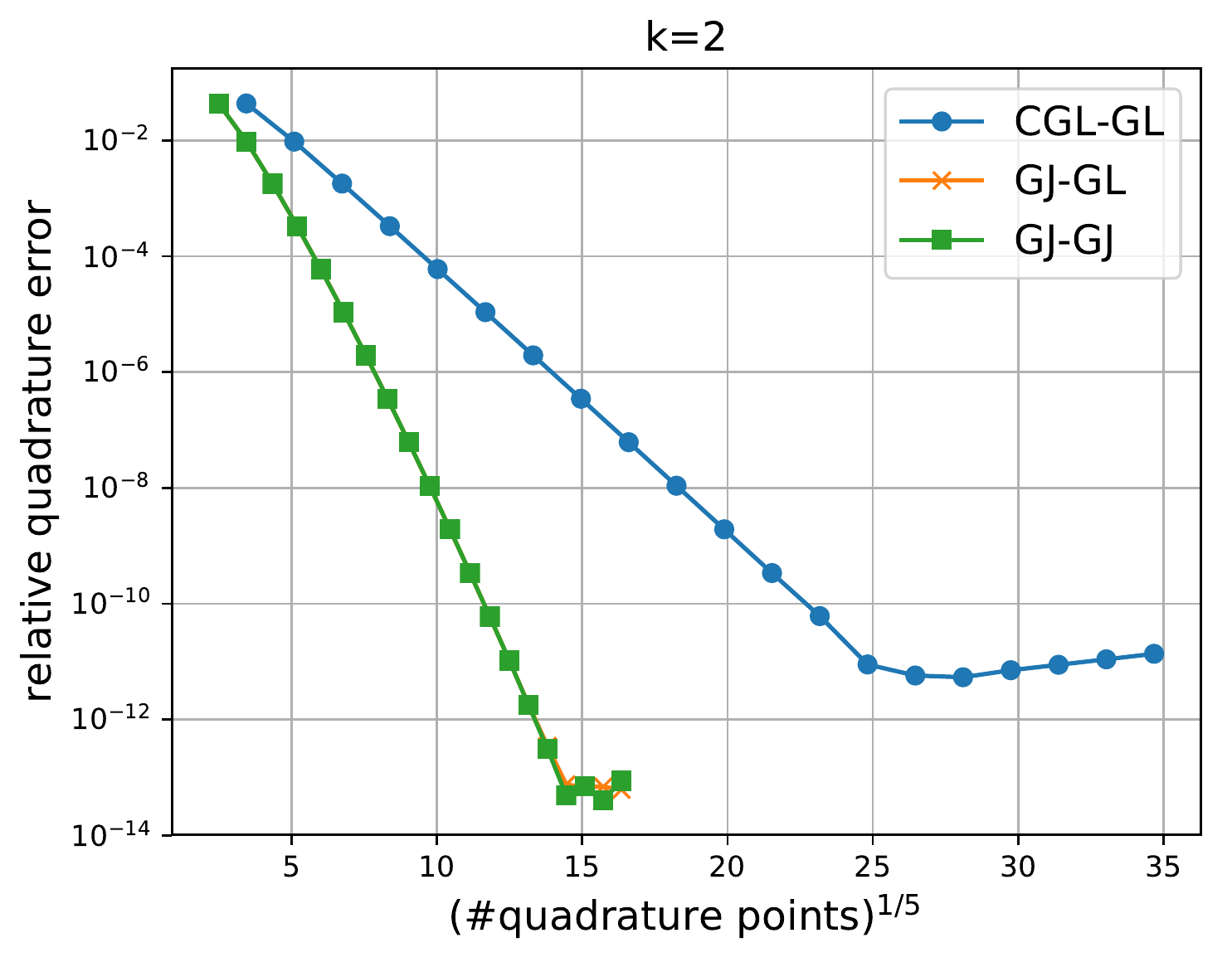}&\\
    (e)&\\
  \end{tabular}
  \caption{%
  Verification of the exponential convergence rate of the quadratures~\cite{chernov2015quadrature},
  which we use for computing elemental stiffness matrices; \(s=1/3\) in this simulation.
  \(k\) is the dimension of the intersection \(\bar{T}_1 \cap \bar{T}_2\);
  \(k=-1\) means \(\bar{T}_1 \cap \bar{T}_2=\emptyset\) and therefore the integrand
  is regular; see \cite{chernov2015quadrature} for details.
  (a): \(k=-1\) and \(\diam(\bar{T}_1 \cup \bar{T}_2)<\delta\);
  (b): \(k=-1\) and \(\diam(\bar{T}_1 \cup \bar{T}_2)>\delta\);
  (c): \(k=0\); (d): \(k=1\); (e): \(k=2\).%
  }
  \label{fig:integration}
\end{figure}

Because of such a high cost of elemental integration, and because the number of
integrals in a quasi-uniform grid grows as \(O(\delta^n h^{-2n})\), we focus
on regular grids (see Fig.~\ref{fig:mesh}).
In this setting we only need to evaluate integrals for a fixed ``reference''
\(T_1\) and varying \(T_2\), thereby bringing the number of integrals down to
\(O(\delta^n h^{-n})\) as shown in Fig.~\ref{fig:mesh2}.
Even with this preprocessing, both the work and memory requirements for the global
matrix assembly scale as \(O(\delta^n h^{-2n})\).
Putting this into perspective, for Grid2 with \(h=2^{1/2}/320\) (i.e., each side
of the unit square \(\O\) is discretized with \(320\) elements) we need approximately
\(4.7\)Mb to store the precomputed integrals and approximately \(3.87\cdot 10^3\)Mb to store
the assembled matrix.
Direct solver such as UMFPACK~\cite{davis2004algorithm} quickly run out of memory
for problems with \(\delta=0.1\), and we switch to CG-accelerated Ruge--Stuben
AMG solver PyAMG~\cite{OlSc2018} (even smoothed aggregation is too much memory
and computationally demanding).
\subsubsection{\(h\)-convergence test}
In order to test the code, we use the method of manufactured solutions, see
e.g.~\cite{roache2002code}.
We put \(\delta=0.1\), \(s\in\{1/3,2/3\}\), \(\kappa\equiv 1\), and let the analytical solution
to be \(u_{\text{ana}}(x,y) = [x(1-x)y(1-y)]^2 \sin(2\pi(x+y^2))\), when \(x,y \in \O\), and
zero otherwise.
The corresponding right hand side can be (numerically) computed as
\begin{equation}\label{eq:anarhs}
  f(x,y) = -2\lim_{\epsilon\to 0}\int_{\epsilon}^\delta \int_{0}^{2\pi}
  A_{\delta}(r)\frac{u_{\text{ana}}(x+r\cos(\theta),y+r\sin(\theta))-u_{\text{ana}}(x,y)}{r}
  \,\mathrm{d}\theta\,\mathrm{d}r,
\end{equation}
which is evaluated using the standard adaptive quadrature package in SciPy.
The results of this test are shown in Fig.~\ref{fig:hconv}.
In both cases we do observe convergence, although it is difficult to say
whether we reach the asymptotic convergence rate (we run the simulations until
we run out of memory), or indeed whether the adaptive quadratures
evaluate~\eqref{eq:anarhs} accurately enough.
\begin{figure}
  \centering
  \includegraphics[width=0.49\columnwidth]{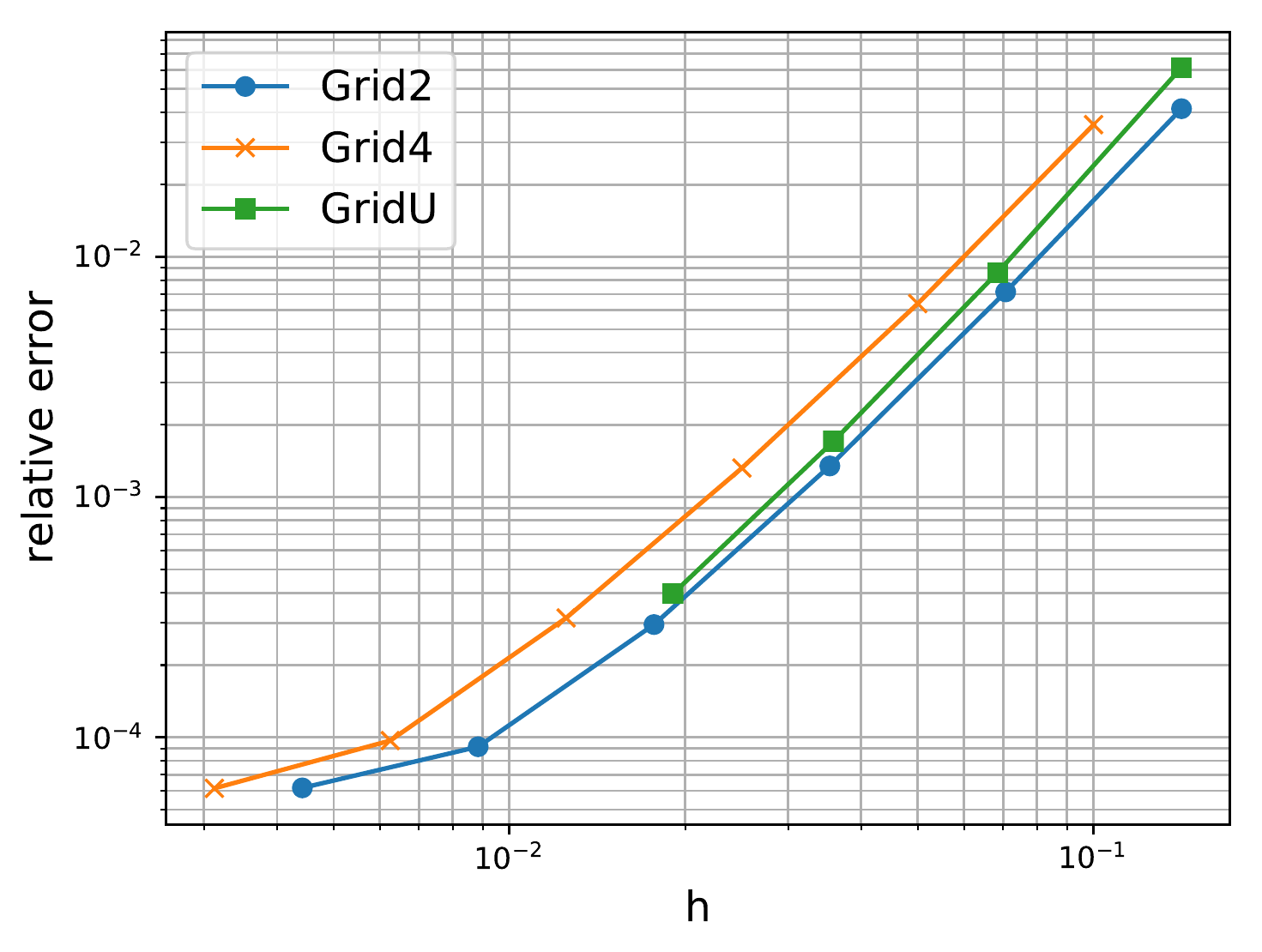}
  \includegraphics[width=0.49\columnwidth]{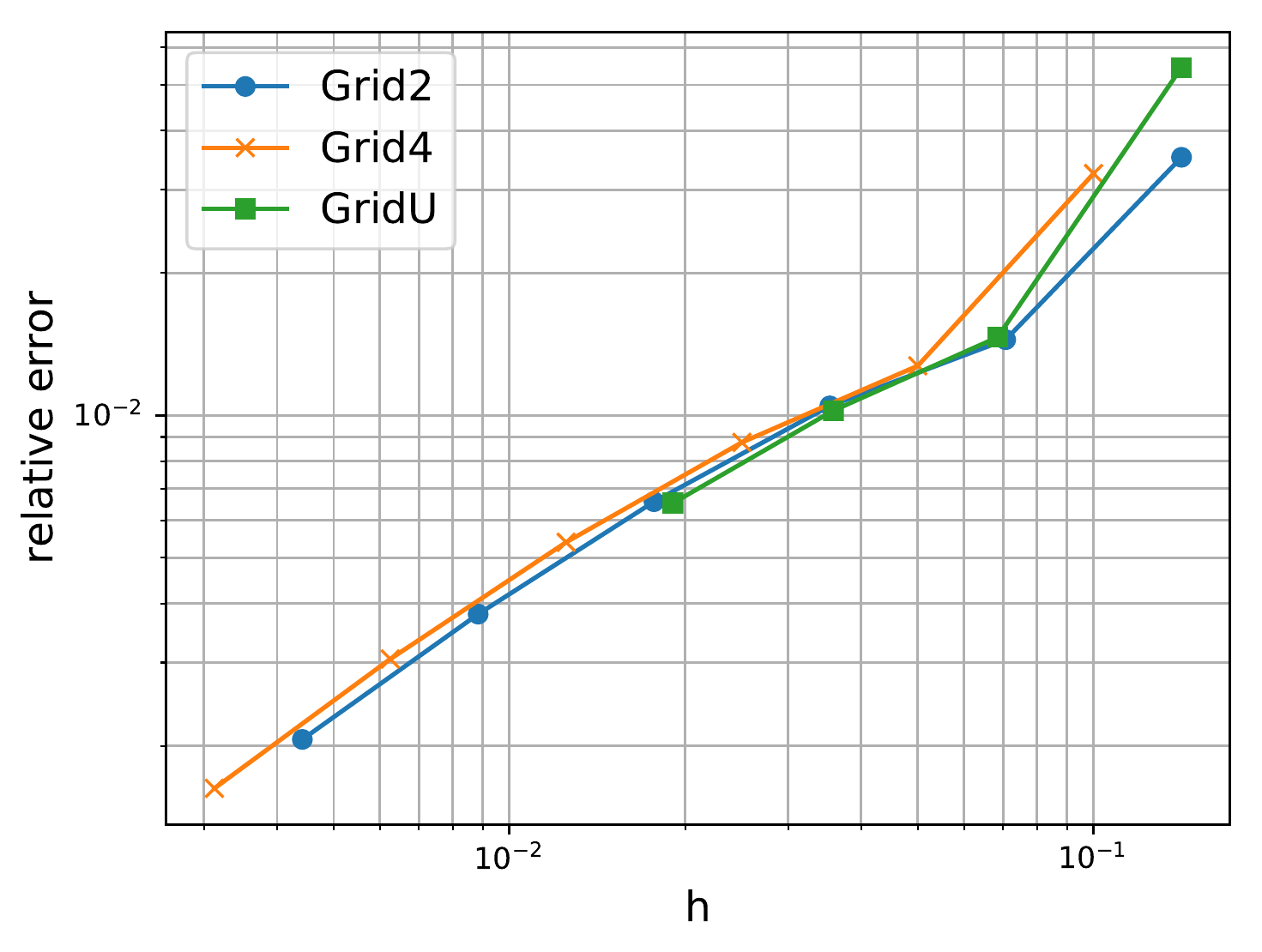}
  \caption{%
  Verification of convergence of the discretization (and the numerical implementation)
  with respect to \(h\)-refinement.
  Left: \(s=1/3\), right: \(s=2/3\).
  The relative error is measured with respect to \(L^2(\Od)\) norm.
  %, and
  %the asymptotic convergence rate is estimated to be \(O(h^{0.84})\) for
  %Grid2 and \(O(h^{0.86})\) for Grid4.%
  }
  \label{fig:hconv}
\end{figure}
\subsubsection{``\(\Gamma\)-convergence'' test}
Another test we perform is that of ``numerical \(\Gamma\)-convergence'',
that is, we try to illustrate Proposition~\ref{prop:limit}.
Namely we put \(s=1/3\),
\(\rho(x,y)= \underline{\rho} +(\overline{\rho}-\underline{\rho}) \exp\{-[(x-m_x)^2+(y-m_y)^2]/\sigma\}\)
with \((m_x,m_y) = (1/2,2/3)\) and \(\sigma = 0.1\), and
\(\kappa\) computed from \(\rho\) using the SIMP model with \(p=2\),
\(u_{\text{ana}}(x,y) = \sin(2\pi x) \sin(\pi y)\), for \(x,y \in \O\), and
zero otherwise.
We compute \(f = -\diver[\rho^p \nabla u_{\text{ana}}]\), and solve a variety
of non-local problems with varying \(\delta\) and \(h\) on Grid2.
The results are summarized in Fig.~\ref{fig:deltaconv}.
The main observation is that the quantity
\(e(\delta)= \|u_{\text{ana}}-\lim_{h\to 0} u_{\delta,h}\|_{L^2(\O)}\),
where \(u_{\delta,h}\) is the numerical solution to the discretized non-local
problem~\eqref{eq:wh}, appears to be decreasing with \(\delta\).
Unsurprisingly, one needs finer resolution meshes to resolve
non-local problems with smaller values of \(\delta\).
\begin{figure}
  \centering
  \includegraphics[width=0.49\columnwidth]{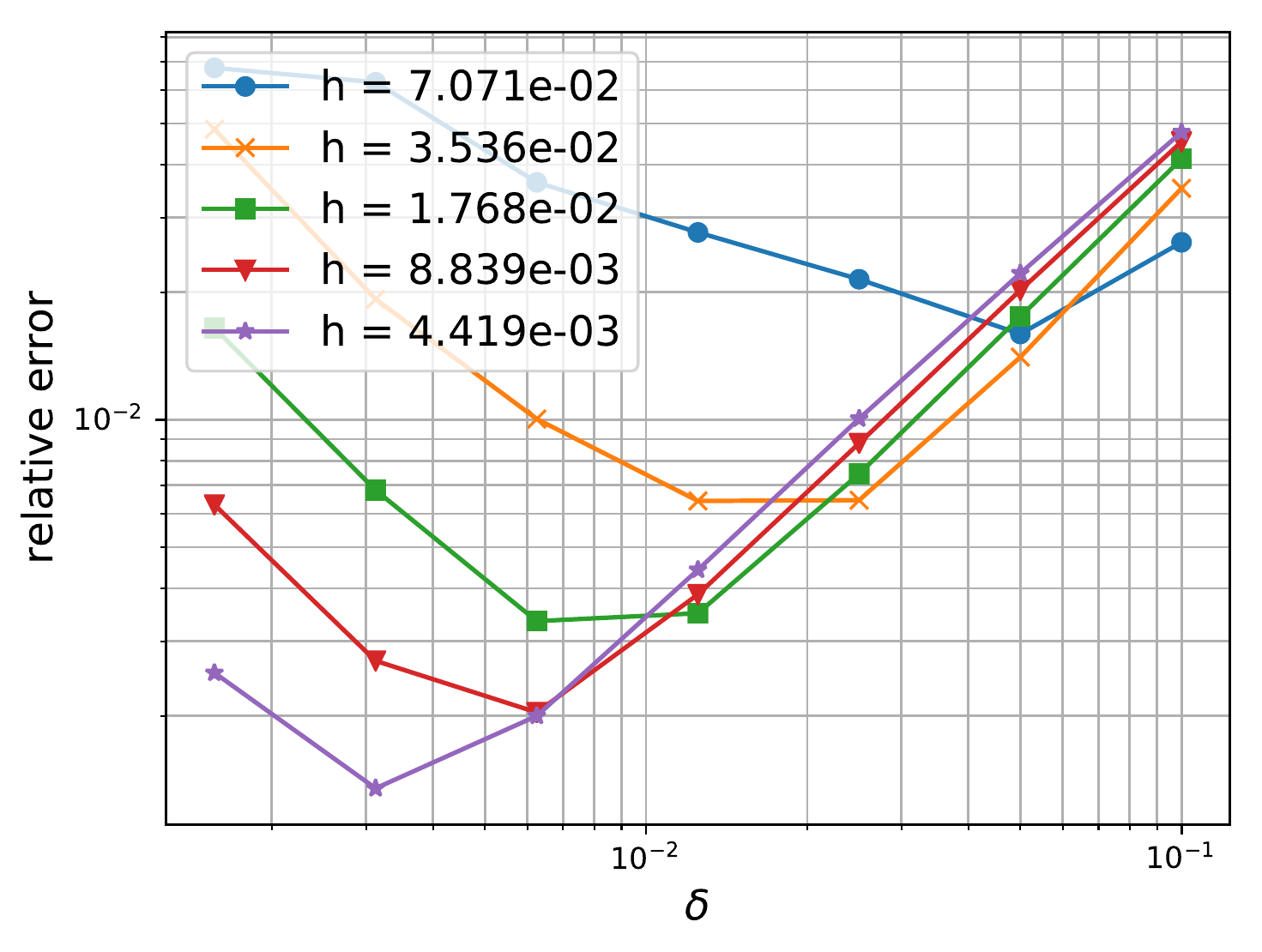}
  \caption{%
  Numerical verification of \(\Gamma\)-convergence.
  The error is measured as \(L^2(\O)\)-norm of the difference between the
  numerical solution of the non-local problem and the analytical solution
  of the local problem.%
  }
  \label{fig:deltaconv}
\end{figure}
\subsection{Solving the optimization problem}

In our implementation we solve a discretized version of the optimization problem~\eqref{eq:min}
with a tiny discrepancy: instead of the constraint \(\int_{\O} \rho(x)\dx \le \gamma|\O|\)
in~\eqref{eq:adm} we have implemented the constraint \(\int_{\Od} \rho(x)\dx \le \gamma|\O|\).
We focus exclusively on compliance minimization, that is, \(J(\rho,u)=\ell(u)\),
and we employ the so-called optimality criterion (OC) scheme for solving
these problems.
This choice is primarily owing to the popularity of OC in the topology
optimization community, and any other gradient-based non-linear constrained
optimization algorithm could be utilized in its place.
Within the OC scheme, given a current material distribution \(\rho_k\) we first
compute the corresponding state \(u_{\rho_k}\) by solving the discretized system~\eqref{eq:wh},
and then the derivatives of non-local compliance \(\compnl\) as
\begin{equation*}
  \begin{aligned}
  \compnl'(\rho_k;\xi) = -\frac{p}{2}\int_{\Od}\int_{\Od}
  [\rho_k^{p/2-1}(x)\rho_k^{p/2}(x')\xi(x)
  +\rho_k^{p/2}(x)\rho_k^{p/2-1}(x')\xi(x')]
  A_{\delta}(|x-x'|)\frac{(u_{\rho_k}(x)-u_{\rho_k}(x'))^2}{|x-x'|^2}\dx'\dx.
 \end{aligned}
\end{equation*}
The new material distribution is defined by a simple pointwise update scheme
\begin{equation*}
  \rho_{k+1} = \pi_{B_k}[\rho_k (-\nabla \compnl(\rho_k)/\lambda_{k+1})^\xi],
\end{equation*}
where \(\pi_{B_k}[\cdot]\) is a projection operator  onto a closed, convex, and
non-empty set \(B_k = \{\,
\rho \in L^2(\Od) \mid \max[\underline{\rho},(1-\eta)\rho_k] \leq \rho \leq
\min[\overline{\rho},(1+\eta)\rho_k]\,\}\), and \(\eta>0\) and \(\xi \in (0,1)\)
are trust-region like and damping parameters, respectively,
and \(\nabla \compnl(\rho_k)\) is \(L^2(\Omega_\delta)\) representation of
the directional derivatives \(\compnl'(\rho_k;\xi)\).
Finally, \(\lambda_{k+1}\) is computed by finding the root of the equation
\(\int_{\Od} \rho_{k+1}(x) \dx = \gamma|\O|\) using, for example, the bisection
algorithm.
We put \(\eta = 0.2\), \(\xi = 0.5\), and stop the algorithm when
\(\|\rho_{k+1}-\rho_k\|_{L^2(\Od)} < 10^{-4}\).
For more details see~\cite{bendose2003topology}.
\subsubsection{Convex case: \(p=1\)}
In the the ``easy'' convex case corresponding to \(p=1\) we can expect that the
optimization algorithm computes approximations to (discretized approximations of)
globally optimal solutions.
The computed local conductivities \(\kappa^{\text{loc}}=\rho\) and the corresponding
states for several values of \(\delta\) are shown in Figure~\ref{fig:p1opt}.
Note that in this case the local compliance minimization problem~\eqref{eq:min_loc}
admits globally optimal solutions, which is also shown in Figure~\ref{fig:p1opt}
with the corresponding state.
The qualitative resemblance between the solutions to the non-local and local problems
is clear from these pictures.
Additionally, we seem to have a quantitative connection between the two problems,
illustrated in Table~\ref{tbl:p1}.

Note that the number of optimization iterations needed to solve the problem
is virtually independent from the value of the non-local horizon \(\delta\) or
size of the mesh \(h\).
Still, each iteration of the optimization algorithm, which requires solving the
discretized state equations, becomes significantly more costly for larger \(\delta>0\).
\begin{table}
  \centering
  \begin{tabular}{rrrr}
    \(\delta\) & \(h\) & \(J^*\) & \(N\)\\
    \hline
    \(0.2\) & \(3.54\cdot 10^{-2}\) & \(9.0727 \cdot 10^{-2}\) & \(26\)\\
    \(0.2\) & \(1.77\cdot 10^{-2}\) & \(9.1778 \cdot 10^{-2}\) & \(27\)\\
    \(0.2\) & \(8.84\cdot 10^{-3}\) & \(9.2408 \cdot 10^{-2}\) & \(27\)\\
    \(0.2\) & \(4.42\cdot 10^{-3}\) & \(9.2798 \cdot 10^{-2}\) & \(27\)\\
    \hline
    \(0.1\) & \(3.54\cdot 10^{-2}\) & \(7.8627 \cdot 10^{-2}\) & \(28\)\\
    \(0.1\) & \(1.77\cdot 10^{-2}\) & \(7.9479 \cdot 10^{-2}\) & \(32\)\\
    \(0.1\) & \(8.84\cdot 10^{-3}\) & \(7.9930 \cdot 10^{-2}\) & \(32\)\\
    \(0.1\) & \(4.42\cdot 10^{-3}\) & \(8.0218 \cdot 10^{-2}\) & \(32\)\\
    \hline
    \(0.05\) & \(3.54\cdot 10^{-2}\) & \(7.3251 \cdot 10^{-2}\) & \(29\)\\
    \(0.05\) & \(1.77\cdot 10^{-2}\) & \(7.3896 \cdot 10^{-2}\) & \(34\)\\
    \(0.05\) & \(8.84\cdot 10^{-3}\) & \(7.4276 \cdot 10^{-2}\) & \(36\)\\
    \(0.05\) & \(4.42\cdot 10^{-3}\) & \(7.4484 \cdot 10^{-2}\) & \(35\)\\
    \hline
    \(0\) & \(3.13\cdot 10^{-3}\) & \(6.9460 \cdot 10^{-2}\) & \(38\)\\
    \hline
  \end{tabular}
  \caption{Summary of the results for the convex case \(p=1\).  \(J^*\): optimal value;
  \(N\): number of OC iterations.  \(\delta=0\) corresponds to the limiting
  local model (solved using a separate code) which admits optimal solutions
  in the convex case \(p=1\), see~\cite{cea1970example,bendose2003topology,allaire2012shape}.}
  \label{tbl:p1}
\end{table}
\begin{figure}
  \centering
  \includegraphics[width=0.32\columnwidth]{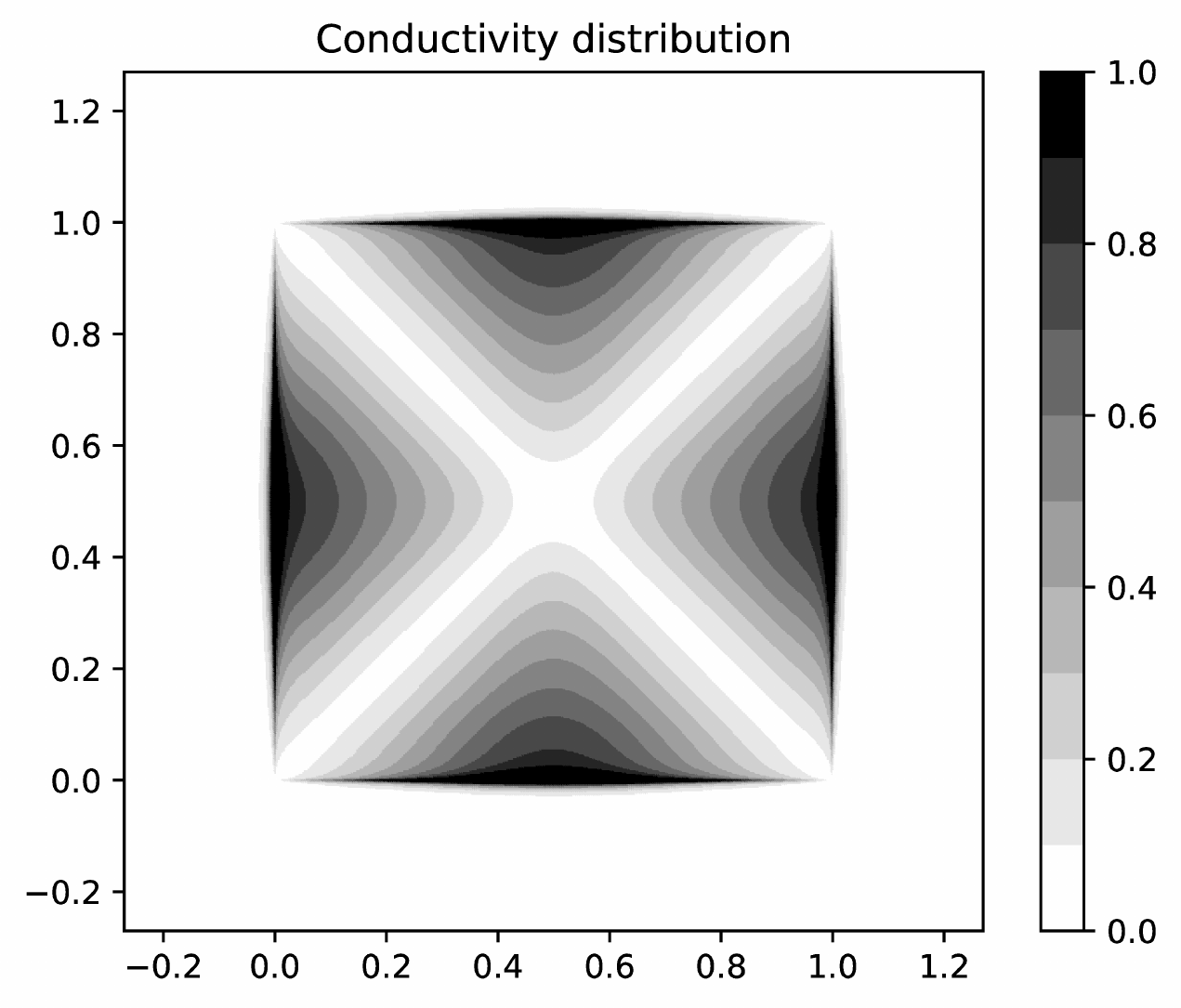}
  \includegraphics[width=0.32\columnwidth]{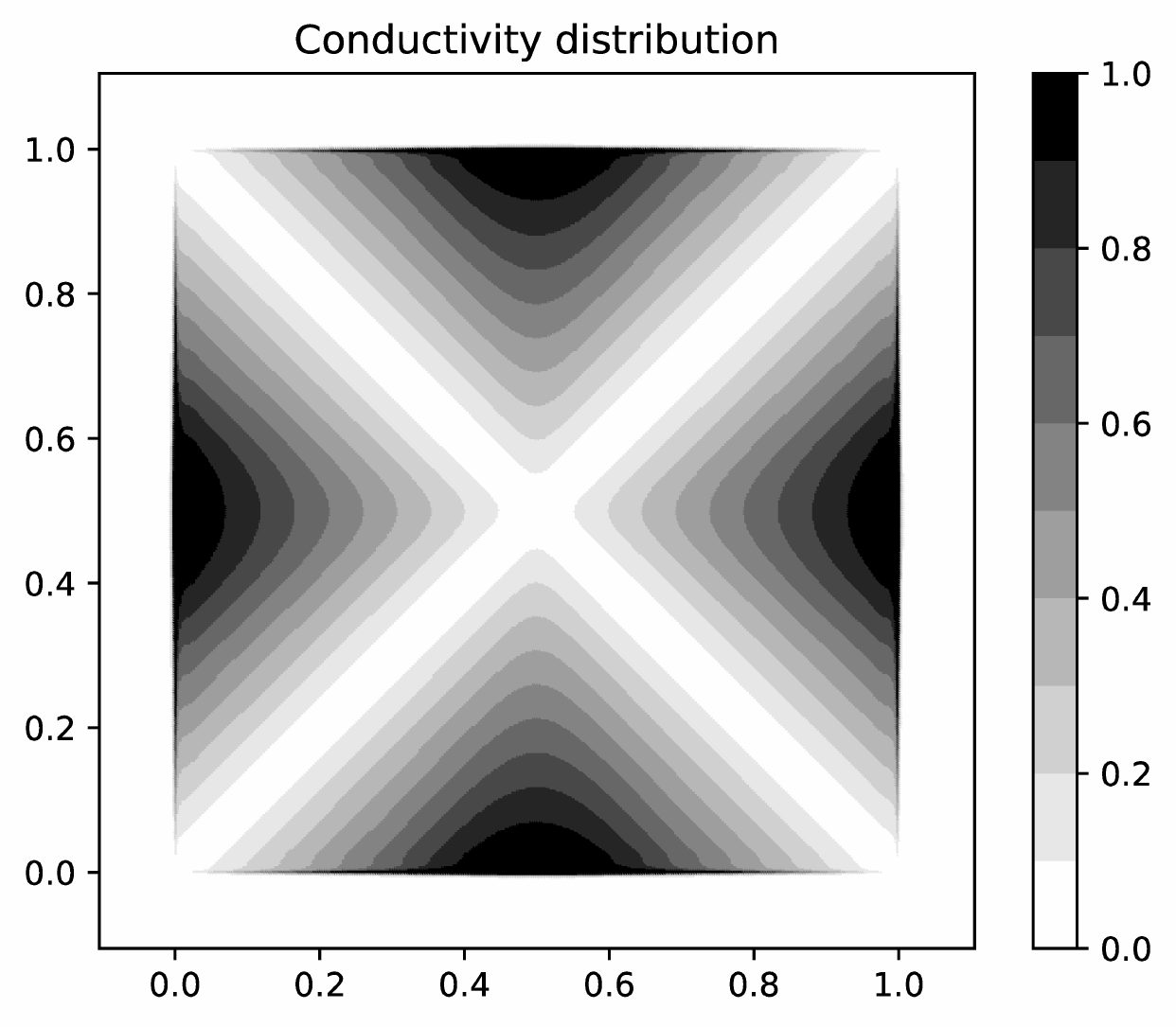}
  \includegraphics[width=0.32\columnwidth]{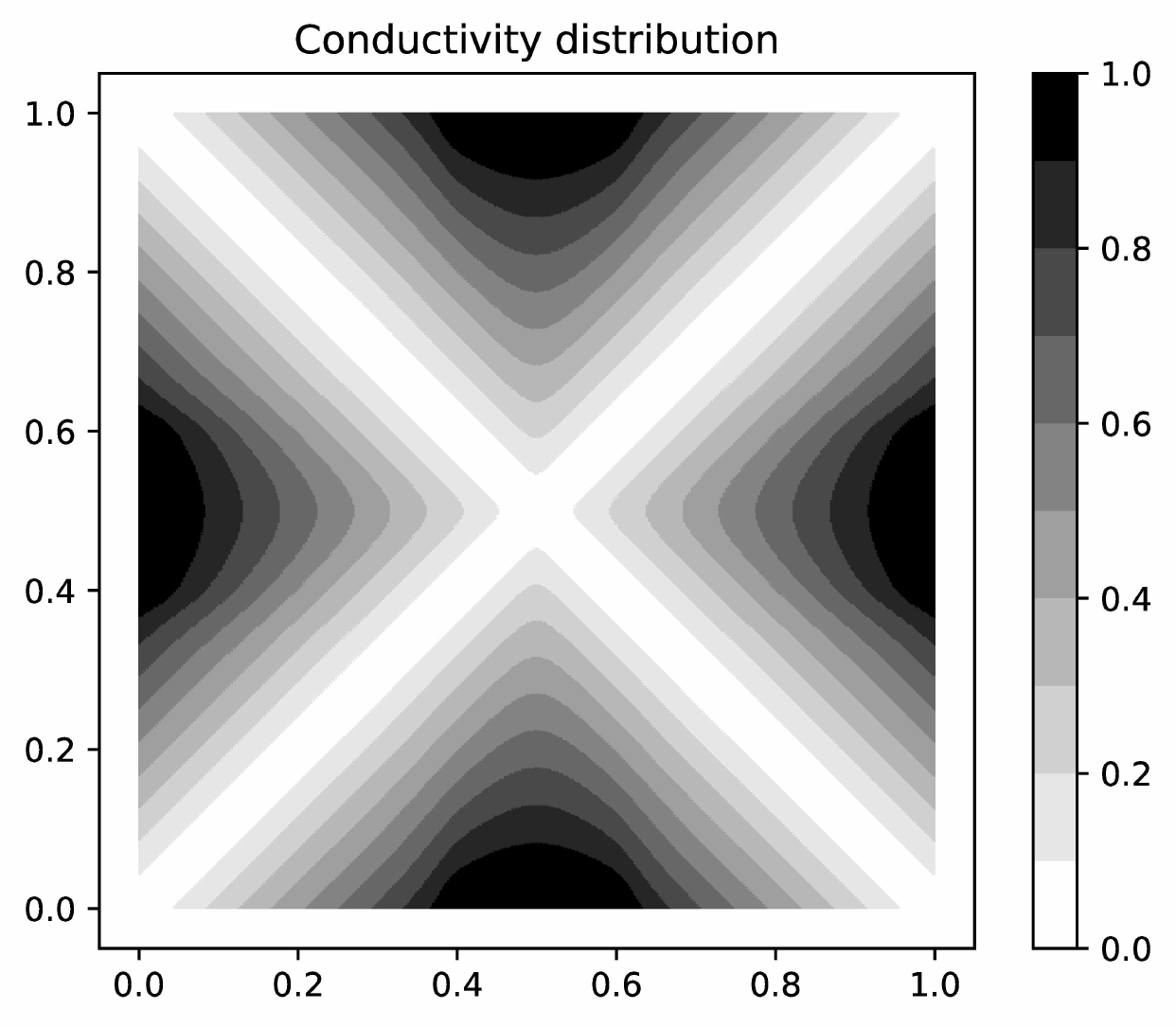}\\
  \includegraphics[width=0.32\columnwidth]{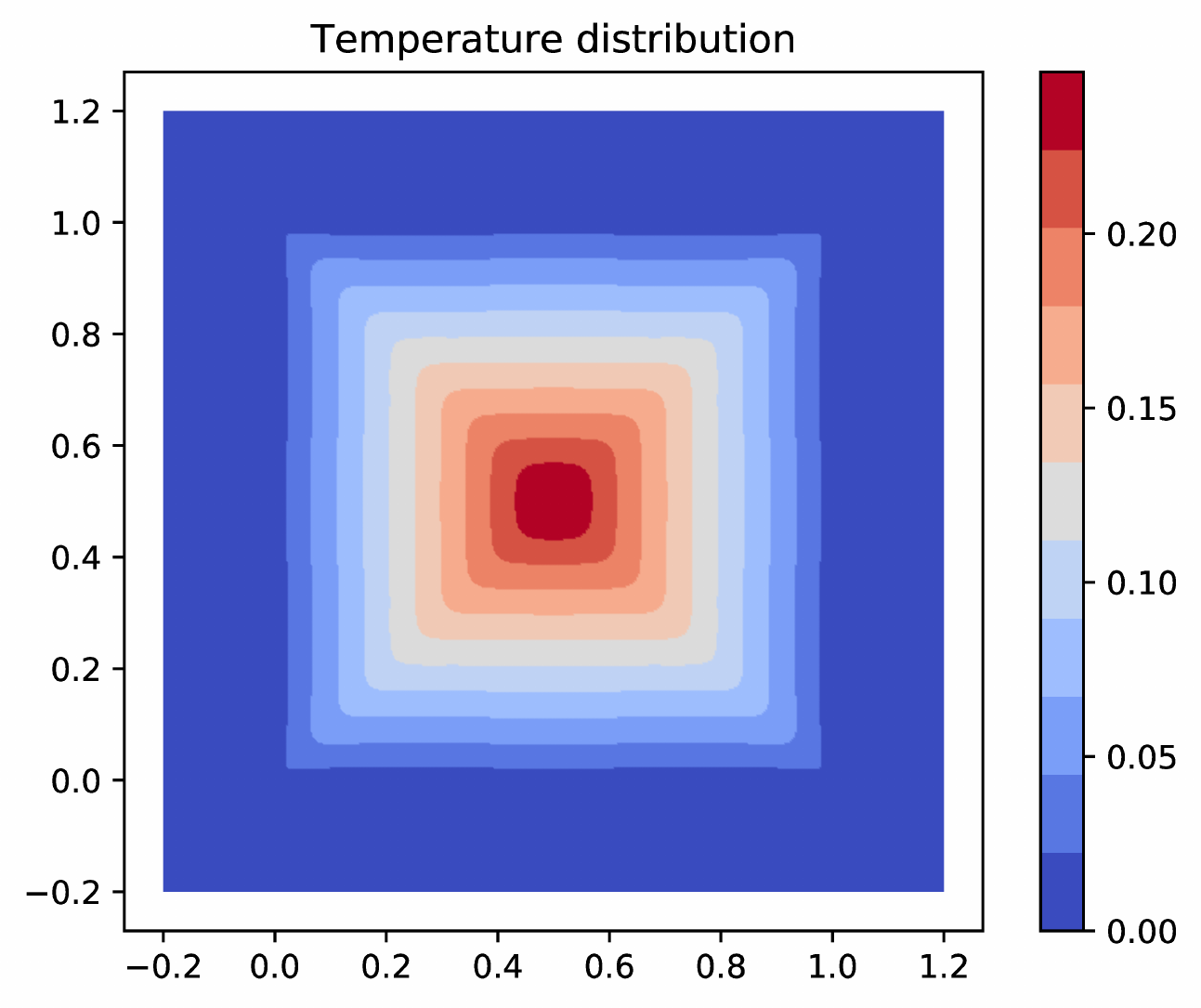}
  \includegraphics[width=0.32\columnwidth]{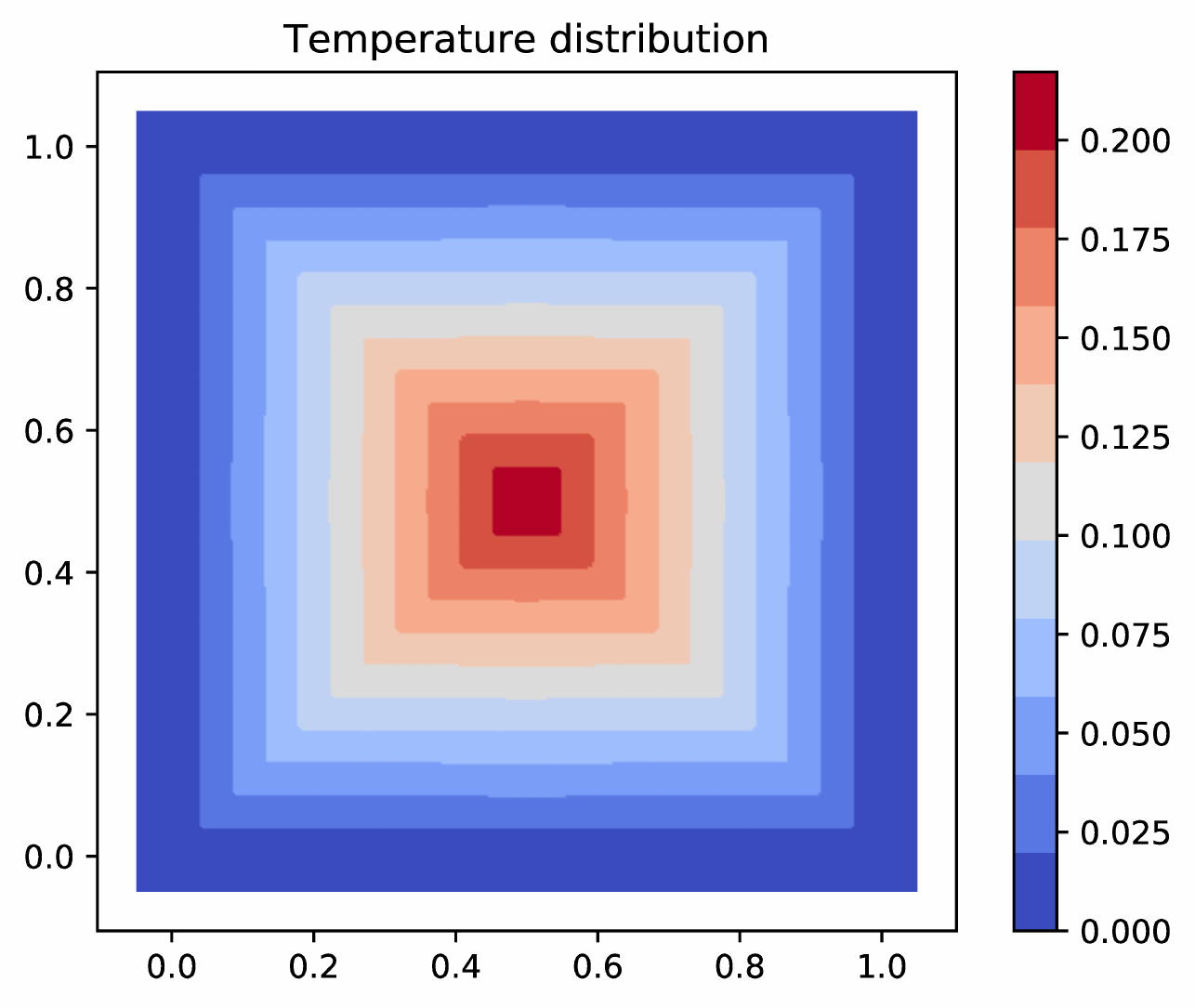}
  \includegraphics[width=0.32\columnwidth]{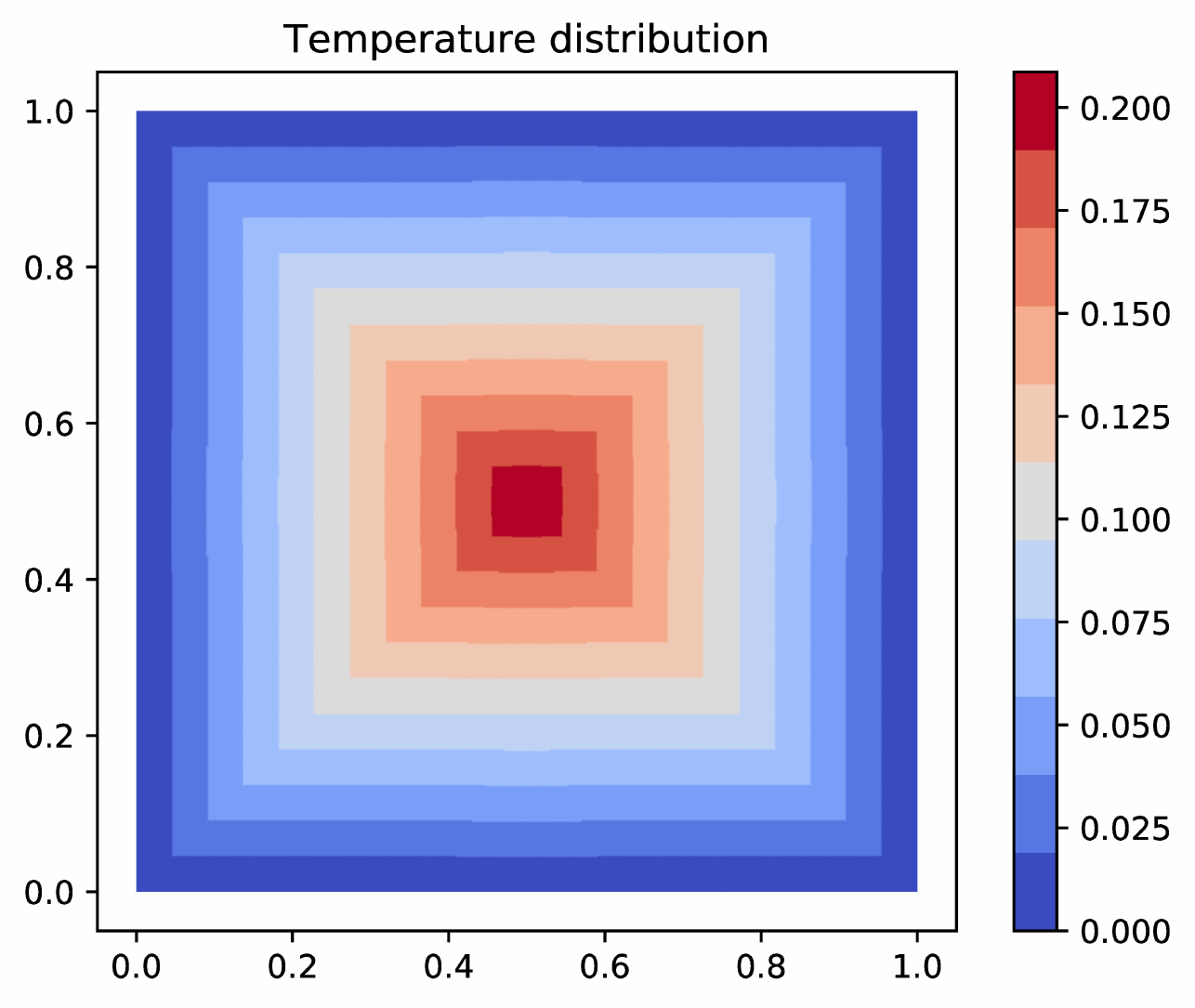}
  \caption{Pictures of the computed optimal designs for \(p=1\).
  Top line: optimal design (conductivity distribution \(\kappa^{\text{loc}}=\rho\)), bottom line:
  corresponding optimal states.
  From left to right: \(\delta=0.2\), \(\delta=0.05\), \(\delta=0\) (local problem).
  Note the slightly different color range and the presence of \(\Gnl\) for the nonlocal
  problems.}
  \label{fig:p1opt}
\end{figure}
\subsubsection{Nonconvex case: \(p=2\)}
A significantly more interesting case, in view of the very different behaviour of the
non-local and local problems, corresponds to the non-convex optimization problem~\eqref{eq:min}
with \(p=2\).
In this case, for the local problem the intermediate values of the conductivity are
effectively penalized by the underlying physics and one can expect the computed
optimized conductivity distribitions to be of ``bang-bang'' structure,
assuming either the lowest or the highest possible values of conductivity everywhere.
Note also that in this case the local optimization problem~\eqref{eq:min_loc} does not admit
optimal solutions, and therefore we have no local solution to compare with.

The computed local conductivities \(\kappa^{\text{loc}}=\rho^p\) and the corresponding
states for several values of \(\delta\) are shown in Figure~\ref{fig:p2opt}.
They agree with our expectations.
One can also note that for smaller \(\delta\) the computed conductivity distributions display
smaller features and progressively more oscillatory character.
Indeed, in the case of the local problem, minimizing sequences consist of conductivities,
which are locally highly oscillatory (periodic) and can be mathematically understood
as converging to composite materials, see~\cite{bendose2003topology,allaire2012shape}.
Quantitative information about these solutions is shown in Table~\ref{tbl:p2}.
Note that for smaller \(\delta\) the optimization algorithm requires significantly more
iterations to converge (which does not mean that it is more difficult to solve, as
each iteration is less computationally expensive in this case).

The non-existence of optimal solutions for the local compliance minimization
problem manifests itself numerically as ``mesh-dependence'' of optimal designs,
where progressively more oscillatory conductivity distributions are encountered
as one refines the computational mesh.
This is often used as an ``engineering'' test of existence of optimal solutions
for a given problem.
The optimal conductivity distributions corresponding to \(\delta=0.2\) and
a range of discretization levels is shown in Figure~\ref{fig:p2meshindep}.
Indeed, one can recognize the sequence of convergent shapes as the discretization
gets finer and finer.
Again, this example works by pure coincidence: indeed as the problem is non-convex
the optimization algorithm may end up in different local minima at different discretization
levels, and indeed we have observed such behavior for other values of \(\delta\).

Finally, we perform a ``cross-check'' of the computed non-local designs by evaluating
all computed designs for all values of \(\delta\).
We expect (but of course cannot guarantee this, as we can only hope to find
locally optimal solutions and not global ones) that the distribution optimized
for a specific value of \(\delta\) would outperform other designs computed for
different values of \(\delta\).
The results of this test are shown in Table~\ref{tbl:p2_crosscheck}.
Indeed, our expectations are confirmed: within each row the diagonal element
is the smallest one.
\begin{figure}
  \centering
  \includegraphics[width=0.32\columnwidth]{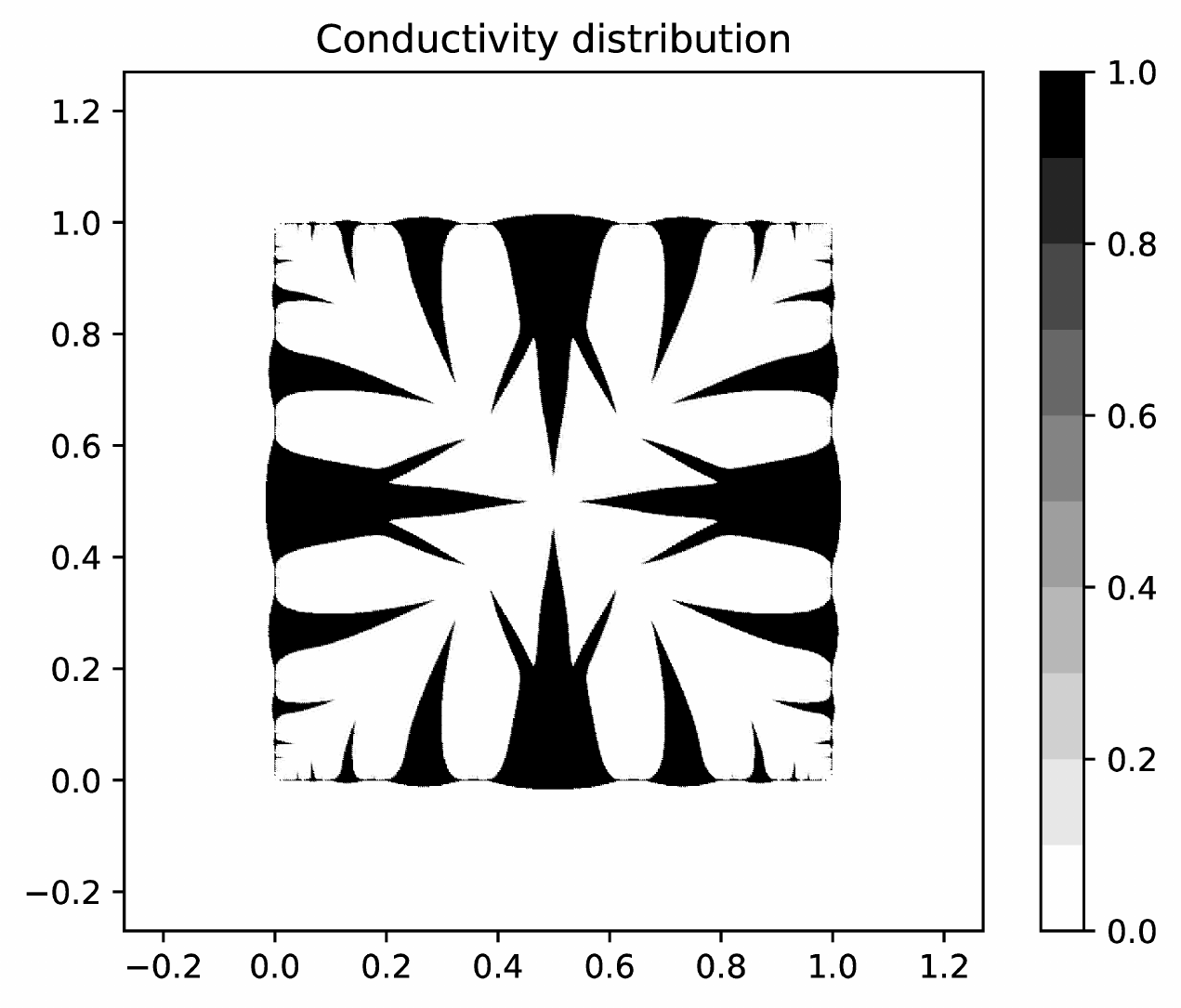}
  \includegraphics[width=0.32\columnwidth]{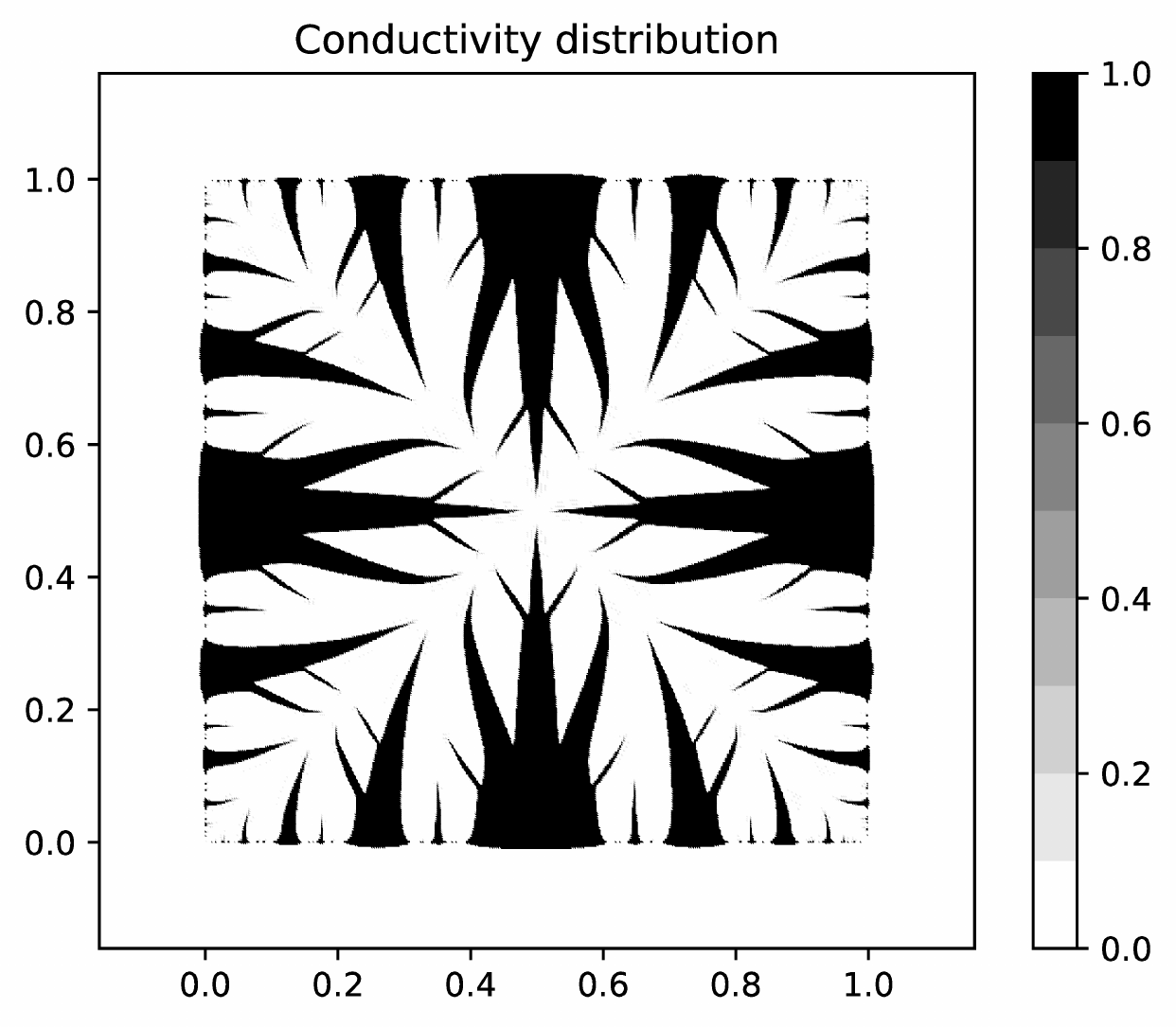}
  \includegraphics[width=0.32\columnwidth]{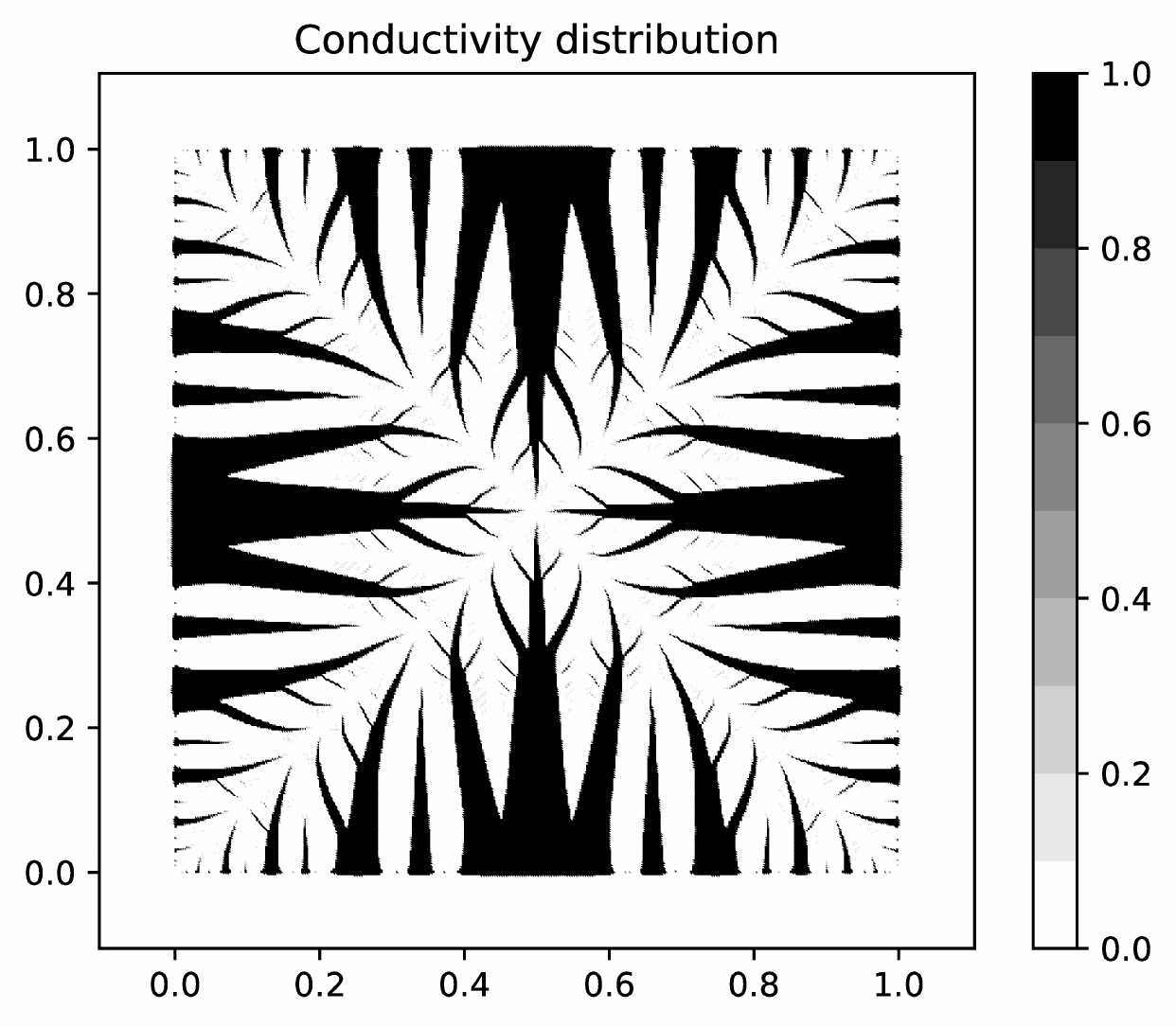}\\
  \includegraphics[width=0.32\columnwidth]{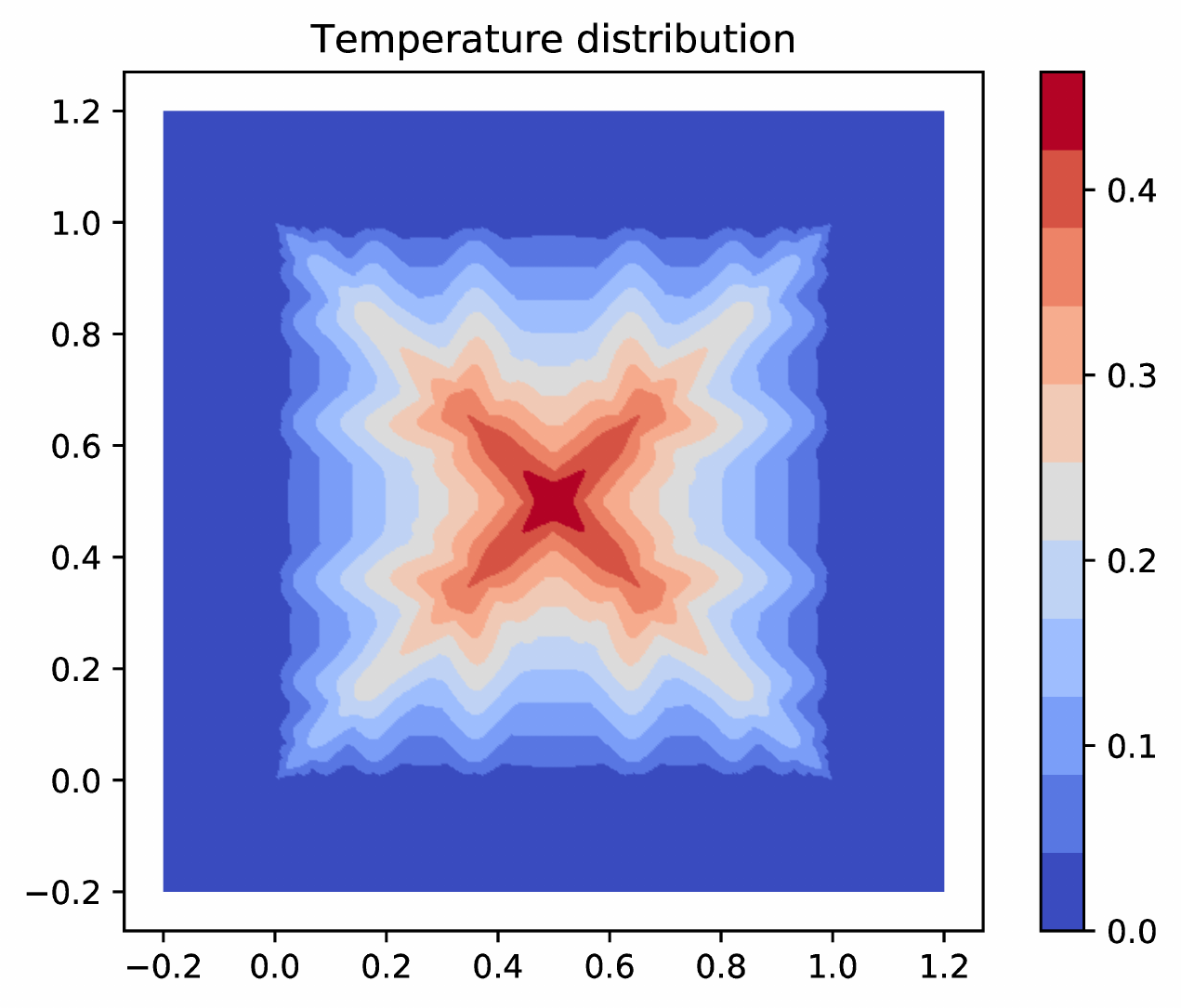}
  \includegraphics[width=0.32\columnwidth]{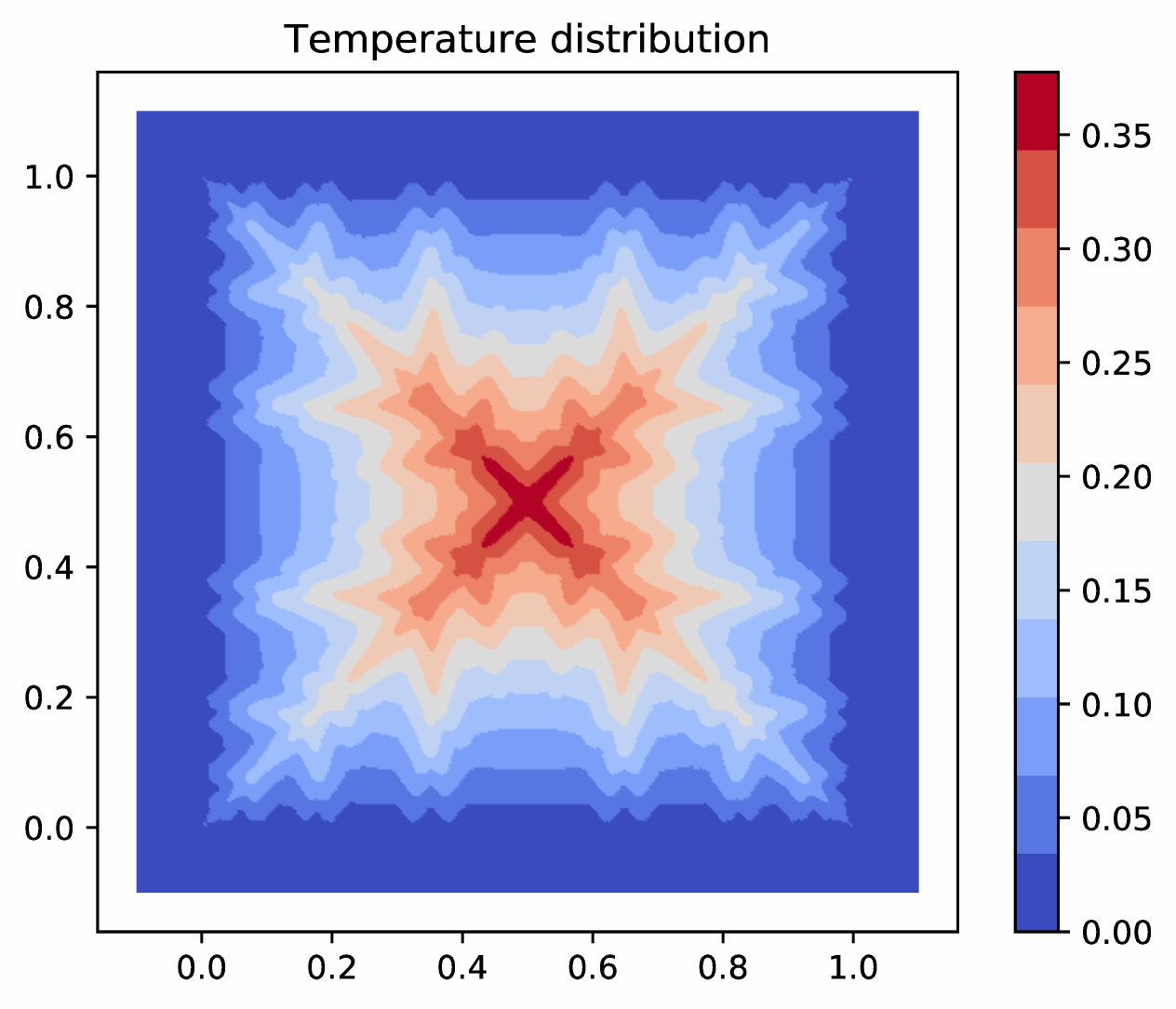}
  \includegraphics[width=0.32\columnwidth]{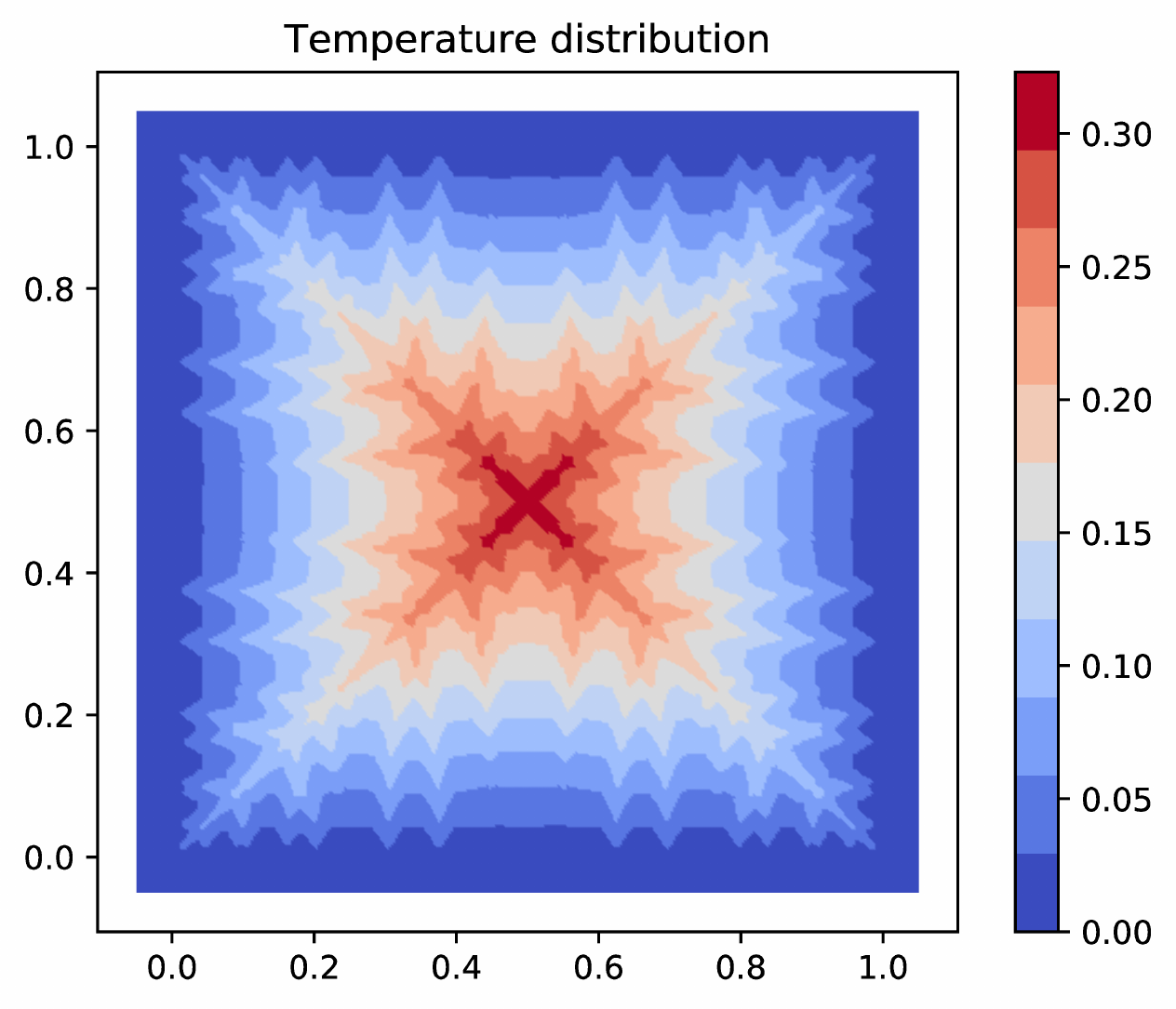}\\
  \caption{Pictures of the computed optimal designs for \(p=2\).
  Top line: optimal design (conductivity distribution \(\kappa^{\text{loc}}=\rho^2\)), bottom line:
  corresponding optimal states.
  From left to right: \(\delta=0.2\), \(\delta=0.1\), \(\delta=0.05\).
  Note the slightly different color range and the presence of \(\Gnl\) for the nonlocal
  problems.}
  \label{fig:p2opt}
\end{figure}
\begin{figure}
  \centering
  \includegraphics[width=0.4\columnwidth]{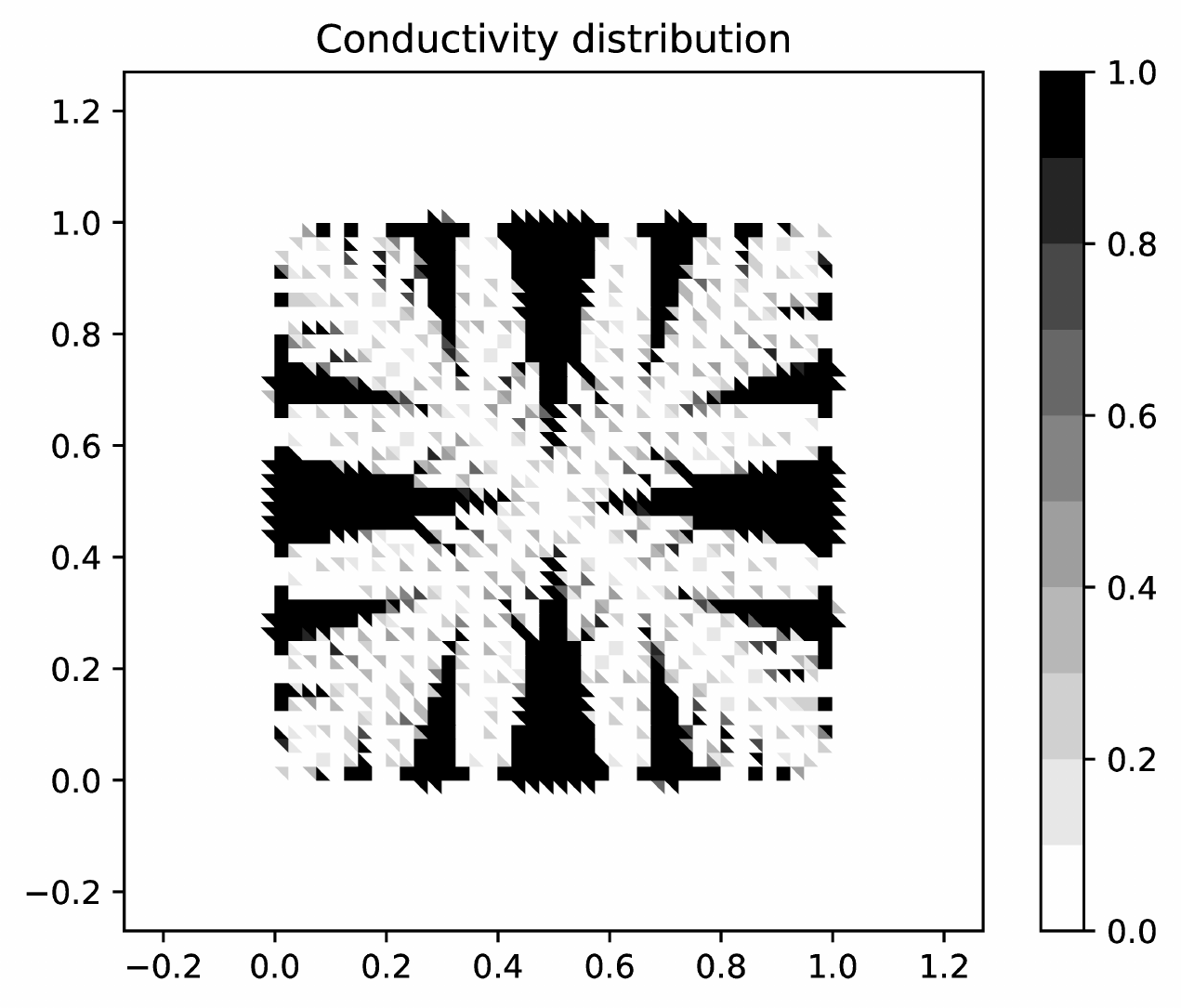}
  \includegraphics[width=0.4\columnwidth]{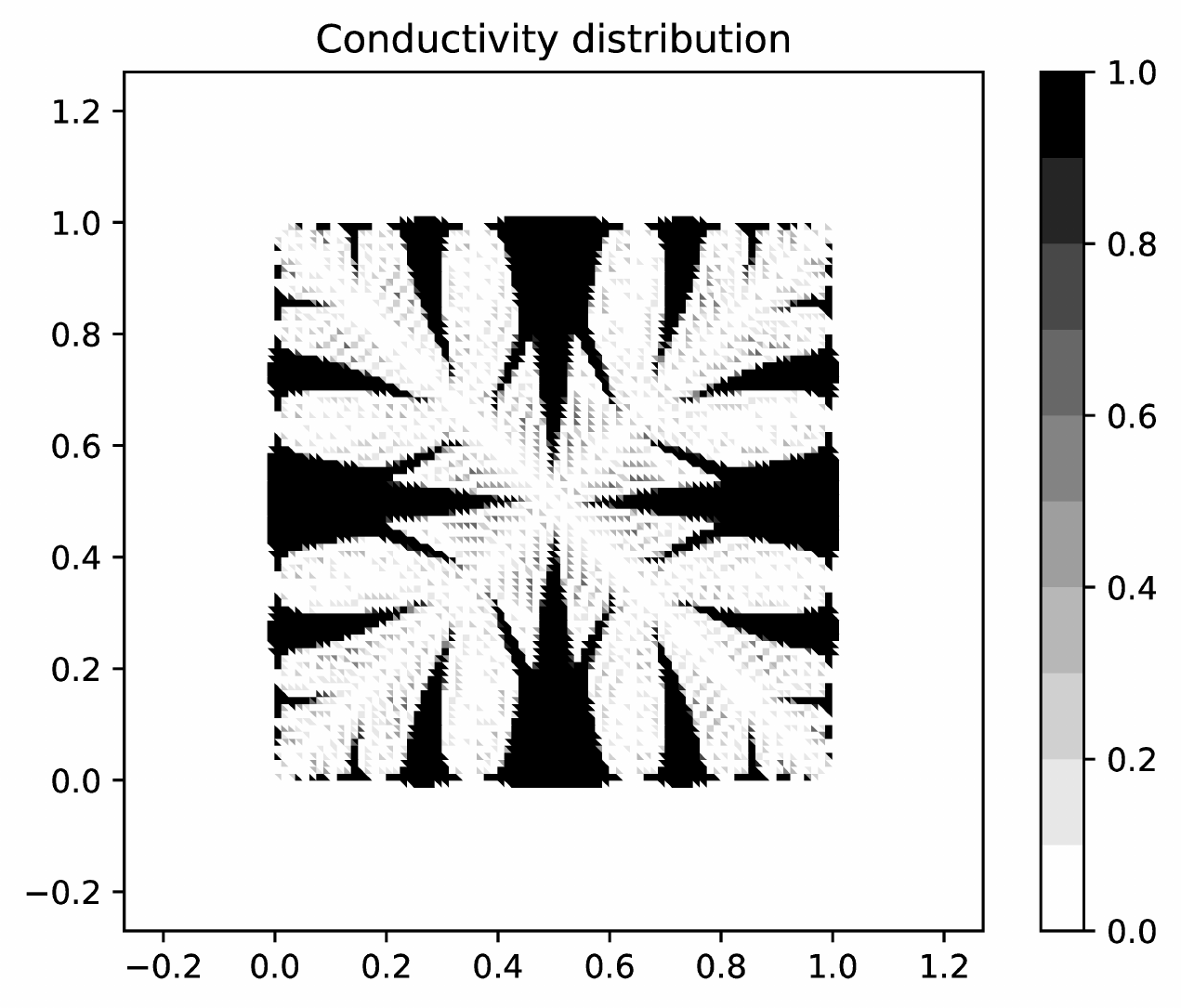}\\
  \includegraphics[width=0.4\columnwidth]{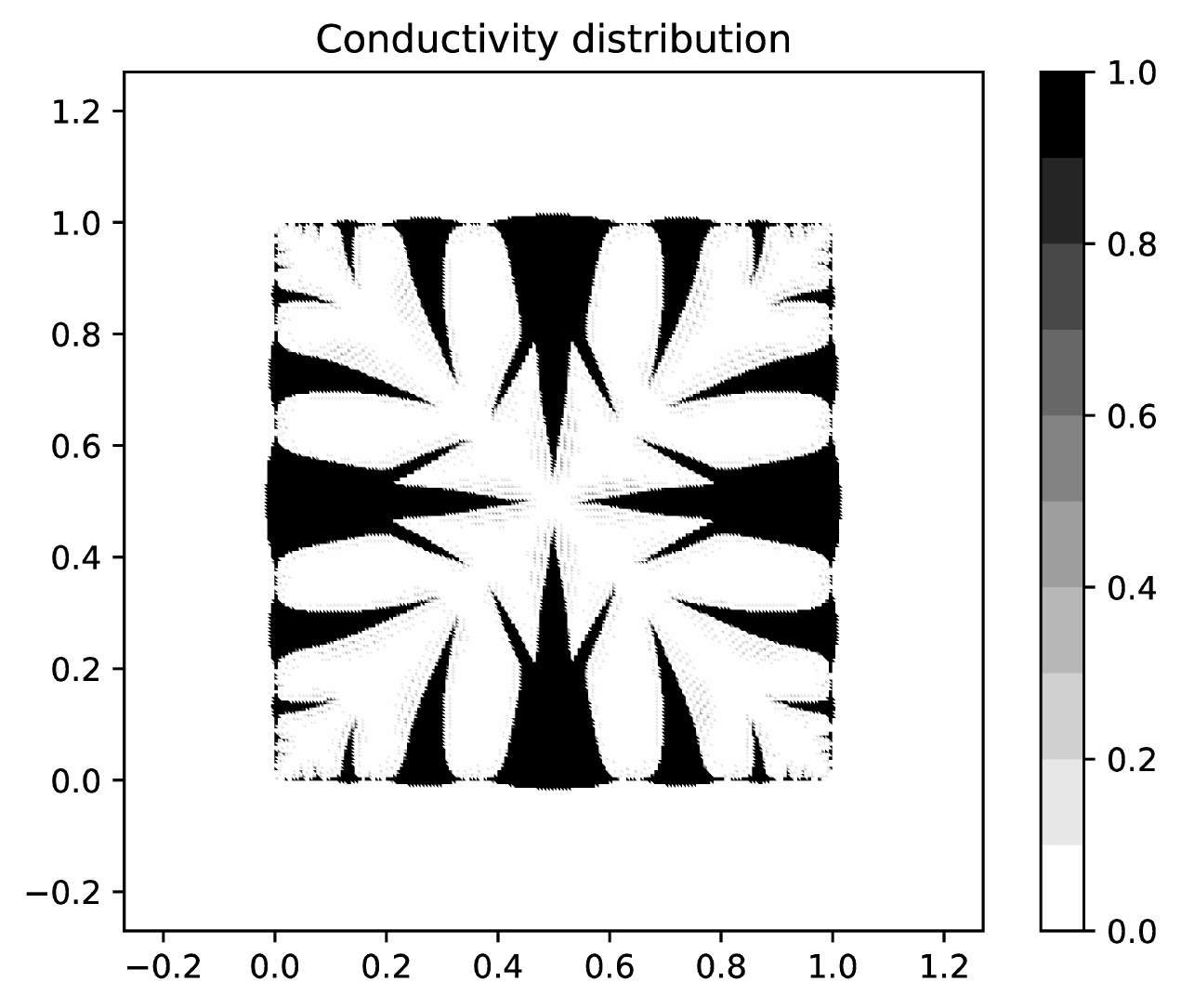}
  \includegraphics[width=0.4\columnwidth]{./NUMERICS/nloc/simp_nl_02_320_rho_01139.pdf}\\
  \caption{``Mesh-independence'' test (owing to the non-convexity of the problem we
  cannot guarantee that the optimization algorithm converges to the ``same'' local minimum
  on different discretizations).
  Conductivity distributions \(\kappa^{\text{loc}}=\rho^2\) are shown.
  We put \(p=2\), \(\delta=0.2\), and \(h \in \{3.54\cdot 10^{-2},1.77\cdot 10^{-2},
  8.84\cdot 10^{-3},4.42\cdot 10^{-3}\}\).%
  }
  \label{fig:p2meshindep}
\end{figure}
\begin{table}
  \centering
  \begin{tabular}{rrrr}
    \(\delta\) & \(h\) & \(J^*\) & \(N\)\\
    \hline
    \(0.2\) & \(3.54\cdot 10^{-2}\) & \(1.6902 \cdot 10^{-1}\) & \(1080\)\\
    \(0.2\) & \(1.77\cdot 10^{-2}\) & \(1.7327 \cdot 10^{-1}\) & \(1062\)\\
    \(0.2\) & \(8.84\cdot 10^{-3}\) & \(1.7516 \cdot 10^{-1}\) & \(1300\)\\
    \(0.2\) & \(4.42\cdot 10^{-3}\) & \(1.7613 \cdot 10^{-1}\) & \(1139\) \\
    \hline
    \(0.1\) & \(3.54\cdot 10^{-2}\) & \(1.2857 \cdot 10^{-1}\) & \(514\)\\
    \(0.1\) & \(1.77\cdot 10^{-2}\) & \(1.3242 \cdot 10^{-1}\) & \(1116\)\\
    \(0.1\) & \(8.84\cdot 10^{-3}\) & \(1.3480 \cdot 10^{-1}\) & \(1395\)\\
    \(0.1\) & \(4.42\cdot 10^{-3}\) & \(1.3560 \cdot 10^{-1}\) & \(1521\)\\
    \hline
    \(0.05\) & \(3.54\cdot 10^{-2}\) & \(1.0721 \cdot 10^{-1}\) & \(261\)\\
    \(0.05\) & \(1.77\cdot 10^{-2}\) & \(1.1066 \cdot 10^{-1}\) & \(790\)\\
    \(0.05\) & \(8.84\cdot 10^{-3}\) & \(1.1139 \cdot 10^{-1}\) & \(1358\)\\
    \(0.05\) & \(4.42\cdot 10^{-3}\) & \(1.1185 \cdot 10^{-1}\) & \(2394\) \\
    \hline
  \end{tabular}
  \caption{Summary of the results for the non-convex case \(p=2\).  \(J^*\): computed objective value;
  \(N\): number of OC iterations.}
  \label{tbl:p2}
\end{table}
\begin{table}
  \centering
  \begin{tabular}{rrrr}
    Original \(\delta\) & \(J\) for \(\delta=0.05\) & \(J\) for \(\delta=0.1\) & \(J\) for \(\delta=0.2\)\\
    \hline
    \(0.05\) &  \(\mathbf{1.1185\cdot 10^{-1}}\) & \(1.1724\cdot 10^{-1}\) & \(1.3225\cdot 10^{-1}\)\\
    \(0.1\)  & \(1.3954\cdot 10^{-1}\) & \(\mathbf{1.3560\cdot 10^{-1}}\) & \(1.4172\cdot 10^{-1}\)\\
    \(0.2\)  & \(1.9678\cdot 10^{-1}\) & \(1.8213\cdot 10^{-1}\) & \(\mathbf{1.7613\cdot 10^{-1}}\)\\
    \hline
  \end{tabular}
  \caption{``Cross-checking'' of the computed designs, \(h=4.42\cdot 10^{-3}\) in all cases.}
  \label{tbl:p2_crosscheck}
\end{table}

%
%
%\section{Conclusions}
%

\section*{Acknowledgements}
AE's research is financially supported by the Villum Fonden through the Villum Investigator Project InnoTop.  The work of JCB is funded by FEDER EU and Ministerio de Econom{\'i}a y Competitividad (Spain) through grant MTM2017-83740-P.

\bibliography{NonlocalOD}

\end{document}